\documentclass[english]{article}
\usepackage[T1]{fontenc}
\usepackage[utf8]{luainputenc}
\usepackage{geometry}
\usepackage{amsthm}
\usepackage{amsmath}
\usepackage{amssymb}
\usepackage{graphicx}
\PassOptionsToPackage{normalem}{ulem}
\usepackage{ulem}

\makeatletter

\theoremstyle{plain}
\newtheorem{thm}{\protect\theoremname}[section]
  \theoremstyle{definition}
  \newtheorem{defn}[thm]{\protect\definitionname}
  \theoremstyle{definition}
  \newtheorem{example}[thm]{\protect\examplename}
  \theoremstyle{remark}
  \newtheorem{rem}[thm]{\protect\remarkname}
  \theoremstyle{plain}
  \newtheorem{cor}[thm]{\protect\corollaryname}
  \theoremstyle{plain}
  \newtheorem{lem}[thm]{\protect\lemmaname}

\usepackage[vcentermath,enableskew]{youngtab}
\usepackage{young}
\usepackage{ytableau}
\usepackage{tikz}
\usepackage{subcaption}
\usepackage{hyperref}

\makeatother

\usepackage{babel}
  \providecommand{\corollaryname}{Corollary}
  \providecommand{\definitionname}{Definition}
  \providecommand{\examplename}{Example}
  \providecommand{\lemmaname}{Lemma}
  \providecommand{\remarkname}{Remark}
\providecommand{\theoremname}{Theorem}

\begin{document}
\global\long\def\pp#1{\mathbf{P}\left(#1\right)}
\global\long\def\ee#1{\mathbf{E}\left[#1\right]}
\global\long\def\norm#1{\left\Vert #1\right\Vert }
\global\long\def\abs#1{\left|#1\right|}
\global\long\def\given#1{\left|\phantom{\frac{}{}}#1\right.}
\global\long\def\ceil#1{\left\lceil #1\right\rceil }
\global\long\def\floor#1{\left\lfloor #1\right\rfloor }
\global\long\def\var{\mathbf{Var}}
\global\long\def\q{\mathbf{Q}}
\global\long\def\cov{\mathbf{Cov}}
\global\long\def\corr{\mathbf{Corr}}
\global\long\def\e{\mathbf{E}}
\global\long\def\one{\mathtt{1}}
\global\long\def\p{\mathbf{P}}

\global\long\def\bA{\mathbb{A}}
\global\long\def\bB{\mathbb{B}}
\global\long\def\bC{\mathbb{C}}
\global\long\def\bD{\mathbb{D}}
\global\long\def\bE{\mathbb{E}}
\global\long\def\bF{\mathbb{F}}
\global\long\def\bG{\mathbb{G}}
\global\long\def\bH{\mathbb{H}}
\global\long\def\bI{\mathbb{I}}
\global\long\def\bJ{\mathbb{J}}
\global\long\def\bK{\mathbb{K}}
\global\long\def\bL{\mathbb{L}}
\global\long\def\bM{\mathbb{M}}
\global\long\def\bN{\mathbb{N}}
\global\long\def\bO{\mathbb{O}}
\global\long\def\bP{\mathbb{P}}
\global\long\def\bQ{\mathbb{Q}}
\global\long\def\bR{\mathbb{R}}
\global\long\def\bS{\mathbb{S}}
\global\long\def\bT{\mathbb{T}}
\global\long\def\bU{\mathbb{U}}
\global\long\def\bV{\mathbb{V}}
\global\long\def\bW{\mathbb{W}}
\global\long\def\bX{\mathbb{X}}
\global\long\def\bY{\mathbb{Y}}
\global\long\def\bZ{\mathbb{Z}}

\global\long\def\cA{\mathcal{A}}
\global\long\def\cB{\mathcal{B}}
\global\long\def\cC{\mathcal{C}}
\global\long\def\cD{\mathcal{D}}
\global\long\def\cE{\mathcal{E}}
\global\long\def\cF{\mathcal{F}}
\global\long\def\cG{\mathcal{G}}
\global\long\def\cH{\mathcal{H}}
\global\long\def\cI{\mathcal{I}}
\global\long\def\cJ{\mathcal{J}}
\global\long\def\cK{\mathcal{K}}
\global\long\def\cL{\mathcal{L}}
\global\long\def\cM{\mathcal{M}}
\global\long\def\cN{\mathcal{N}}
\global\long\def\cO{\mathcal{O}}
\global\long\def\cP{\mathcal{P}}
\global\long\def\cQ{\mathcal{Q}}
\global\long\def\cR{\mathcal{R}}
\global\long\def\cS{\mathcal{S}}
\global\long\def\cT{\mathcal{T}}
\global\long\def\cU{\mathcal{U}}
\global\long\def\cV{\mathcal{V}}
\global\long\def\cW{\mathcal{W}}
\global\long\def\cX{\mathcal{X}}
\global\long\def\cY{\mathcal{Y}}
\global\long\def\cZ{\mathcal{Z}}

\global\long\def\sA{\mathscr{A}}
\global\long\def\sB{\mathscr{B}}
\global\long\def\sC{\mathscr{C}}
\global\long\def\sD{\mathscr{D}}
\global\long\def\sE{\mathscr{E}}
\global\long\def\sFA{\mathscr{F}}
\global\long\def\sG{\mathscr{G}}
\global\long\def\sH{\mathscr{H}}
\global\long\def\sI{\mathscr{I}}
\global\long\def\sJ{\mathscr{J}}
\global\long\def\sK{\mathscr{K}}
\global\long\def\sL{\mathscr{L}}
\global\long\def\sM{\mathscr{M}}
\global\long\def\sN{\mathscr{N}}
\global\long\def\sO{\mathscr{O}}
\global\long\def\sP{\mathscr{P}}
\global\long\def\sQ{\mathscr{Q}}
\global\long\def\sR{\mathscr{R}}
\global\long\def\sS{\mathscr{S}}
\global\long\def\sT{\mathscr{T}}
\global\long\def\sU{\mathscr{U}}
\global\long\def\sV{\mathscr{V}}
\global\long\def\sW{\mathscr{W}}
\global\long\def\sX{\mathscr{X}}
\global\long\def\sY{\mathscr{Y}}
\global\long\def\sZ{\mathscr{Z}}

\global\long\def\tr{\text{Tr}}
\global\long\def\re{\text{Re}}
\global\long\def\im{\text{Im}}
\global\long\def\supp{\text{supp}}
\global\long\def\sgn{\text{sgn}}
\global\long\def\d{\text{d}}
\global\long\def\dist{\text{dist}}
\global\long\def\span{\text{span}}
\global\long\def\ran{\text{ran}}
\global\long\def\ball{\text{ball}}
\global\long\def\Dim{\textnormal{Dim}}
\global\long\def\ai{\text{Ai}}
\global\long\def\occ{\text{Occ}}
\global\long\def\sh{\text{sh}}

\global\long\def\To{\Rightarrow}
\global\long\def\half{\frac{1}{2}}
\global\long\def\oo#1{\frac{1}{#1}}
\newcommand{\slfrac}[2]{\left.#1\middle/#2\right.}

\global\long\def\al{\alpha}
\global\long\def\be{\beta}
\global\long\def\ga{\gamma}
\global\long\def\Ga{\Gamma}
\global\long\def\de{\delta}
\global\long\def\De{\Delta}
\global\long\def\ep{\epsilon}
\global\long\def\ze{\zeta}
\global\long\def\et{\eta}
\global\long\def\th{\theta}
\global\long\def\Th{\Theta}
\global\long\def\ka{\kappa}
\global\long\def\la{\lambda}
\global\long\def\La{\Lambda}
\global\long\def\rh{\rho}
\global\long\def\si{\sigma}
\global\long\def\ta{\tau}
\global\long\def\ph{\phi}
\global\long\def\Ph{\Phi}
\global\long\def\vp{\varphi}
\global\long\def\ch{\chi}
\global\long\def\ps{\psi}
\global\long\def\Ps{\Psi}
\global\long\def\om{\omega}
\global\long\def\Om{\Omega}
\global\long\def\Si{\Sigma}

\global\long\def\dequal{\stackrel{d}{=}}
\global\long\def\pto{\stackrel{\p}{\to}}
\global\long\def\asto{\stackrel{\text{a.s.}}{\to}}
\global\long\def\dto{\stackrel{\text{d.}}{\to}}
\global\long\def\ld{\ldots}
\global\long\def\di{\partial}

\global\long\def\shR{\text{Shapes}_{R}}
\global\long\def\shL{\text{Shapes}_{L}}
\global\long\def\shZ{\text{Shape}_{0}}

\global\long\def\siR{\text{Sizes}_{R}}
\global\long\def\siL{\text{Sizes}_{L}}
\global\long\def\siZ{\text{Size}_{0}}

\newcommand{\aaa}{1.3}
\newcommand{\thk}{ {\scriptstyle \slfrac{\th}{k}} }

\title{Decorated Young Tableaux and the Poissonized Robinson-Schensted Process}
\author{Mihai Nica {\footnote{ Courant Institute of Mathematical Sciences, New York University
251 Mercer Street, New York, N.Y. 10012-1185, \href{mailto:nica@cims.nyu.edu}{nica@cims.nyu.edu}}}}
\maketitle
\begin{abstract}
We introduce an object called a decorated Young tableau which can equivalently be viewed as a continuous time trajectory of Young diagrams or as a non-intersecting line ensemble. By a natural extension of the Robinson-Schensted correspondence, we create a random pair of decorated Young tableaux from a Poisson point process in the plane, which we think of as a stochastic process in discrete space and continuous time. By using only elementary techniques and combinatorial properties, we identify this process as a Schur process and show it has the same law as certain non-intersecting Poisson walkers. 

\end{abstract}

\section{Introduction}

The Poissonized Plancherel measure is a one parameter family of measures
on Young diagrams. For fixed $\th$, this is a mixture of the classical
Plancherel measures by Poisson weights. This mixture has nice properties
that make it amenable to analysis, see for instance \cite{Borodin00asymptoticsof}
and \cite{Johansson01discreteorthogonal}. One way this measure is
obtained is to take a unit rate Poisson point process in the square
$[0,\th]\times[0,\th]$, then interpret the collection of points as
a permutation, and finally apply the Robinson-Schensted (RS) correspondence.
The RS correspondence gives a pair of Young tableaux of the same shape.
The law of the shape of the Young tableaux constructed in this way
has the Poissonized Plancherel measure. Other than the shape, the
information inside the tableaux themselves are discarded in this construction.
This construction has many nice properties: for example, by the geometric
construction of the RS correspondence due to Viennot (see for example
\cite{0387950672} for details), this shows that the maximum number
of Poisson points an up-right path can pass through has the distribution
of the length of the first row of the Poissonized Plancherel measure.
One can use this to tackle problems like the longest increasing subsequence
problem.

In this article, we extend the above construction slightly in order
to keep the information in the Young tableaux that are generated by
the RS algorithm; we do not discard the information
in the tableaux. As a result, we get a slightly richer random object
which we call the Poissonized Robinson-Schensted process. This object
can be interpreted in several ways. If one views the object as a continuous
time Young diagram valued stochastic process, then its fixed time marginals are exactly
the Poissonized Plancherel measure. Moreover, the joint distribution
at several times form a Schur process as defined in \cite{Okounkov01correlationfunction}. The proof uses only simple properties of the RS correspondence and elementary probabilistic arguments. The model is defined in Section 2 and its distribution is characterized in Section 3.

We also show that the process itself is a special case of stochastic
dynamics related to Plancherel measure studied in \cite{Borodin_stochasticdynamics}. 
Unlike the construction from \cite{Borodin_stochasticdynamics}, our methods in this article do not rely on machinery from representation theory. Instead, the proof goes by first finding the multi-time distribution in terms of Poisson probability mass functions using elementary techniques from probability and combinatorics. Only after this, we identify this in terms of a Schur process. The derivation of the distribution does not rely on this previous theory.  The connection here allow us to immediately see asymptotics for the
model, in particular it converges to the Airy-2 line ensemble under
the correct scaling. This is discussed in Section 4. 

It is also possible to obtain the Poissonized RS process as a limit of a discrete time Young diagram process in a natural way. Instead of starting with a Poisson point process, one instead starts with a point process on a lattice so that the number of points at each site has a geometric distribution. This model was first considered by Johansson in Section 5 of \cite{JoDPP}, in particular see his Theorem 5.1. Again, the approach we take in this article uses only elementary techniques from probability and combinatorics which is in contrast to the analytical methods used in \cite{JoDPP}. This is discussed in Section 5.

\subsection{Notation and Background}
\label{sec:notation}
We very briefly go over the definitions/notations used here. For more
details, see \cite{StanleyVol2} or \cite{0387950672}.

We denote by $\bY$ the set of Young diagrams. We think of a Young
diagram $\la\in\bY$ as a partition $\la=\left(\la_{1},\la_{2},\ld\right)$
where $\la_{i}$ are weakly decreasing and with finitely many non-zero
entries. We can equivalently think of each $\la\subset\bN^{2}$ as
a collection of stacked unit boxes by $\left(i,j\right)\in\la\iff j\leq\la_{i}$.
We denote by $\abs{\la}=\sum_{i=1}^{n}\la_{i}$ the total number of
boxes, or equivalently the sum of the row lengths. We will sometimes
also consider skew tableaux, which are the collection of boxes one
gets from the difference of two Young diagrams $\la\backslash\mu$.

A standard Young tableau $T$ can be thought of as a Young diagram $\la$
whose boxes have been filled with the numbers $1,2,\ld,\abs{\la}$,
so that the numbers are increasing in any row and in any column. We
call the diagram $\la$ in this case the shape of the tableau, and
denote this by $\sh(T)$. We denote by $T(i,j)$ the entry written
in the box at location $i,j$. We will also use the notation $\dim(\la)$
to denote the number of standard Young tableau of shape $\la$. This
is called the ``dimension'' since this is also the dimension of
the the irreducible representations of the symmetric group $S\left(\abs{\la}\right)$
associated with $\la$. 

In the above notation the Poissonized Plancherel Measure of parameter
$\th^2$ is:
\[
\p_{\th}\left(\la\right)=e^{-\th^{2}}\left(\frac{\th^{\abs{\la}}\dim(\la)}{\abs{\la}!}\right)^{2}.
\]

The Robinson-Schensted (RS) correspondence is a bijection from the
symmetric group $S_{n}$ to pairs of standard Young tableaux of the
same shape of size $\abs{\sh(T)}=n$ (See \cite{StanleyVol2} Section 7.11 for details on this bijection) We will sometimes refer to this
here as the ``ordinary'' RS correspondence, not to diminish the
importance of this, but to avoid confusion with a closely related
map we introduce called the ``decorated RS correspondence''. 

We will also make reference to the Schur symmetric functions $s_{\la}(x_{1},\ld)$,
and the skew Schur symmetric functions $s_{\la/\mu}(x_{1},\ld)$ as
they appear in \cite{StanleyVol2} or \cite{0387950672} . A specialization
is a homomorphism from symmetric functions to complex numbers. We
denote by $f(\rho)$ the image of the function $f$ under the specialization
$\rho$. We denote by $\rho_{t}$ the Plancherel specialization (also
known as exponential or ``pure gamma'' specialization) that has $h_{n}(\rho_{t})=\frac{t{}^{n}}{n!}$
for each $n\in\bN$. This is a Schur positive specialization, in the
sense that $s_{\la}\left(\rho_{t}\right)\geq0$ always, and moreover
there is an explicit formula for $s_{\la}\left(\rho_{t}\right)$ in
terms of the number of Young tableaux of shape $\la$:

\begin{equation} \label{schurid}
s_{\la}(\rho_{t})=\dim(\la)\frac{t^{\abs{\la}}}{\abs{\la}!}.
\end{equation}

\section{Decorated Young Tableaux}
\begin{defn}
\label{decoratedYTdef} A \textbf{decorated} \textbf{Young tableau
}is a pair $\tilde{T}=(T,(t_{1},\ld,t_{\abs{\sh(T)}}))$ where $T$
is a standard Young tableau and $0\leq t_{1}<\ld<t_{\abs{\sh(T)}}$
is an increasing list of non-negative numbers whose length is equal
to the size of the tableau. We refer to the list $(t_{1},\ld,t_{\abs{\sh(T)}})$
as the \textbf{decorations} of the tableau. We represent this graphically
when drawing the tableau by recording the number $t_{T(i,j)}$ in
the box $(i,j)$ . \end{defn}
\begin{example}
\label{YTexample}
The decorated Young Tableau: $$ \left( \young(124,35,6), \left(0.02,0.03,0.05,0.07,0.11,0.13 \right) \right), $$
is represented as:
\begin{center}
\ytableausetup{centertableaux,boxsize=2.5em} 
\begin{ytableau} 
\stackrel{1}{0.02} & \stackrel{2}{0.03} & \stackrel{4}{0.07} \\
\stackrel{3}{0.05} & \stackrel{5}{0.11} \\
\stackrel{6}{0.13} 
\end{ytableau}.
\end{center}\end{example}
\begin{rem}
\label{otherwaytothink}Since the decorations are always sorted, we
see that from the above diagram one could recover the entire decorated
tableau without the labels ``1'', ``2'' written in the tableau.
In other words, one could equally well think of a decorated Young
tableau as a map $\tilde{T}:\sh(T)\to\bR_{+},$ so that $\tilde{T}$
is increasing in each column and in each row. Having $\tilde{T}=(T,(t_{1},\ld,t_{\abs{\sh(T)}}))$
will be slightly more convenient for our explanations here, and particularly
to relate the model to previous work.\end{rem}
\begin{defn}
\label{YDprocess}A decorated Young tableau can also be thought of
as a trajectory of Young \uline{diagrams} evolving in continuous
time. The \textbf{Young diagram process} of the decorated Young tableau
$\tilde{T}=(T,(t_{1},\ld,t_{\abs{\sh(T)}}))$ is a map $\la_{\tilde{T}}:\bR_{+}\to\bY$
defined by
\[
\la_{\tilde{T}}(t)=\left\{ (i,j):\ t_{T(i,j)}\leq t\right\} \in\bY.
\]
One can also think about this as follows: the process starts with
$\la(0)=\emptyset$, and then it gradually adds boxes one by one.
The decoration $t_{T(i,j)}$ is the time at which the box $(i,j)$
is added. The fact that $T$ is a standard Young tableau ensures that
$\la(t)$ is indeed a Young diagram at every time $t$. Notice that
the Young diagram process for a decorated Young tableau is always
increasing $\la(t_{1})\subset\la(t_{2})$ whenever $t_{1}\leq t_{2}$,
and it can only increase by at most one box at a time $\lim_{\ep\to0}\abs{\la(t+\ep)-\la(t)}\leq1$.
Moreover, given any continuous time sequence of Young diagrams evolving
in this way we can recover the decorated Young tableau: if the $k$-th
box added to the sequence is the box $(i,j)$ and it is added at time
$s$, then put $T(i,j)=k$ and $t_{k}=s$ 
\end{defn}

\begin{defn}
\label{nonintersectingLineEnsemble}A decorated Young tableau can
also be thought of as an ensemble of non-intersecting lines. The \textbf{non-intersecting
line ensemble }of the decorated Young tableau $\tilde{T}=(T,(t_{1},\ld,t_{\abs{\sh(\la)}}))$
is a map $M_{\tilde{T}}:\bN\times\bR_{+}\to\bZ$ defined by:
\begin{eqnarray*}
M_{\tilde{T}}(i;t) & = & \la_{i}(t)-i\\
 & = & \abs{\left\{ j:t_{T(i,j)}\leq t\right\} }-i,
\end{eqnarray*}
where $\la_{\tilde{T}}(t)=\left(\la_{1}(t),\ld\right)$ is the Young
diagram process of $\tilde{T}$. The index $i$ is the label of the particle, and the variable $t$ measures the time along the trajectory. The lines $M_{\tilde{T}}(i;t)$ are
non-intersecting in the sense that $M_{\tilde{T}}(i;t)<M_{\tilde{T}}(j;t)$
for $i<j$ and for every $t\in\bR_{+}$. This holds since $\la_{\tilde{T}}(t)\in\bY$
is a Young diagram. It is clear that one can recover the Young diagram
process from the non-intersecting line ensemble by $\la_{i}(t)=M_{\tilde{T}}(i;t)+i$. \end{defn}
\begin{rem}
The map from Young diagrams to collection of integers by $\la\to\left\{ \la_{i}-i\right\} _{i=1}^{\infty}$
is a well known map with mathematical significance, see for instance
\cite{1212.3351} for a survey. This is sometimes presented as the
map $\la\to\left\{ \la_{i}-i+\half\right\} _{i=1}^{\infty}$, where
the target is now half integers. This representation of Young diagrams, which are also known as Maya diagrams, sometimes makes
the resulting calculations much nicer. In this work they do not play
a big role, so we will omit the $\half$ that some other authors use.
\end{rem}

\subsection{Robinson–Schensted Correspondence}
\begin{defn}
\label{associatedpermutation} Fix a parameter $\theta\in\bR_{+}$
and let $\cC_{n}^{\th}$ be the set of $n$ point configurations in
the square $[0,\th]\times[0,\th]\subset\bR_{+}^{2}$ so that no two
points lie in the same horizontal line and no two points lie in the
same vertical line. Every configuration of points $\Pi\in\cC_{n}^{\th}$
has an \textbf{associated permutation} $\si\in S_{n}$ by the following
prescription. Suppose that $0\leq r_{1}<\ld<r_{n}\leq\th$ and $0\leq\ell_{1}<\ld<\ell_{n}\leq\th$
are respectively the sorted lists of $x$ and $y$ coordinates of
the points which form $\Pi$. Then find the unique permutation $\si\in S_{n}$
so that $\Pi=\left\{ \left(r_{i},\ell_{\si(i)}\right)\right\} _{i=1}^{n}$.
Equivalently, if we are given the list of points $\Pi=\left\{ \left(x_{i},y_{i}\right)\right\} _{i=1}^{n}$
sorted so that $0\leq x_{1}<\ld<x_{n}\leq\th$ then $\si$ is the
permutation so that $0\leq y_{\si^{-1}(1)}<y_{\si^{-1}(2)}<\ld<y_{\si^{-1}(n)}\leq\th$.
\end{defn}

\begin{defn}
Let
\[
\cT_{n}^{\th}=\left\{ \left(L,(\ell_{1},\ld,\ell_{n})\right),\left(R,\left(r_{1},\ld,r_{n}\right)\right):\ \sh(L)=\sh(R),0\leq\ell_{1}<\ld<\ell_{n}\leq\th,0\leq r_{1}<\ld<r_{n}\leq\th\right\} 
\]
be the set of pairs of decorated Young tableaux of the same shape
and of size $n$, whose decorations lie in the interval $[0,\th]$. 

The \textbf{decorated Robinson–Schensted (RS) correspondence }is a
bijection $dRS:\cT_{n}^{\th}\to\cC_{n}^{\th}$ from pairs of decorated
Young tableaux in $\cT_{n}^{\th}$ to configuration of points in $\cC_{n}^{\th}$ defined as follows:

Given a pair of decorated Tableau of size $n$, $\left(L,(\ell_{1},\ld,\ell_{n})\right),\left(R,\left(r_{1},\ld,r_{n}\right)\right)$,
use the ordinary RS bijection and the pair of Young tableaux $(L,R)$
to get a permutation $\si\in S_{n}$. Then define
\[
dRS\left(\left(L,(\ell_{1},\ld,\ell_{n})\right),\left(R,\left(r_{1},\ld,r_{n}\right)\right)\right)=\left\{ (r_{1},\ell_{\si(1)}),(r_{2},\ell_{\si(2)}),\ld,(r_{n},\ell_{\si(n)})\right\}. 
\]
Going the other way, the inverse $dRS{}^{-1}:\cC_{n}^{\th}\to\cT_{n}^{\th}$
is described as follows. Given a configuration of points from $\cC_{n}^{\th}$,
first take the permutation $\si$ associated with the configuration
as described in Definition \ref{associatedpermutation}. Then use
the ordinary RS bijection to find a pair of standard Young tableaux
$(L,R)$ corresponding to this permutation. Define
\[
dRS{}^{-1}\left(\left\{ (x_{1},y_{1}),\ld,(x_{n},y_{n})\right\} \right)=\left(L,\left(y_{\si^{-1}(1)},\ld,y_{\si^{-1}(n)}\right),\left(R,(x_{1},\ld,x_{n})\right)\right).
\]

Since the ordinary RS algorithm is a bijection from pairs of standard
Young diagrams of size $n$ to permutations in $S_{n}$, and since
the decorations can be recovered from the coordinates of the points
and vice versa as described above, the decorated RS algorithm is indeed
a bijection as the name suggests. See Figure \ref{fig:small_example} for an example of this bijection. For convenience we will later on use the notations $\mbox{\ensuremath{\cC}}^{\th}=\cup_{n\in\bN}\cC_{n}^{\th}\text{ and }\cT^{\th}=\cup_{n\in\bN}\cT_{n}^{\th}$.\end{defn}
\begin{rem}
With the viewpoint as in Remark \ref{otherwaytothink}, one can equivalently
construct the decorated RS bijection by starting with the list of
points in $\Pi$ in ``two line notation'' $\binom{x_{1}\ x_{2}\ \ld\ x_{n}}{y_{1\ }y_{2}\ \ld\ y_{n}}$,
where the points $\left\{ \left(x_{i},y_{i}\right)\right\} _{i=1}^{n}$
are sorted by $x$-coordinate, and then apply the RS
insertion algorithm on these points to build up the Young tableaux
$\tilde{L}$ and $\tilde{R}$. The same rules for insertion in the
ordinary RS apply; the only difference is that the entries and comparisons
the algorithm makes are between real numbers instead of natural numbers.

Each of the individual decorated tableaux from a pair $\left(\tilde{L},\tilde{R}\right)\in\cT_{n}$
have an associated Young diagram process as defined in Definition\ref{YDprocess}
and an associated non-intersecting line ensemble as defined in Definition
\ref{nonintersectingLineEnsemble}. Since both $L$ and $R$ are the
same shape, and since the decoration are all in the range $[0,\th]$,
the Young diagram processes and the non-intersecting line ensembles
will agree at all times $t\geq\th.$ That is to say $\la_{\tilde{L}}(t)=\la_{\tilde{R}}(t)$
and $M_{\tilde{L}}(\cdot;t)=M_{\tilde{R}}(\cdot;t)$ for $t\geq\th$.
For this reason, it will be more convenient to do a change of coordinates
on the time axis so that the Young diagram process and non-intersecting
line ensemble are defined on $[-\th,\th]$, and the meeting of the
left and right tableau happen at $t=0$. The following definition
makes this precise.\end{rem}
\begin{defn}
\label{YDprocess_pair}For a pair of decorated Young tableaux, we define the \textbf{Young diagram process
$\la_{\tilde{L},\tilde{R}}:[-\th,\th]\to\bY$ }of the pair $\left(\tilde{L},\tilde{R}\right)\in\cT_{n}$
by
\[
\la_{\tilde{L},\tilde{R}}(t)=\begin{cases}
\la_{\tilde{L}}(\th+t) & \ t\leq0\\
\la_{\tilde{R}}(\th-t) & \ t\geq0
\end{cases}.
\]

Notice that this is well defined at $t=0$ since $\la_{\tilde{L}}(\th)=\sh(L)=\sh(R)=\la_{\tilde{R}}(\th)$.
In this way $\la_{\tilde{L},\tilde{R}}$ is an increasing sequence
of Young diagrams when $t<0$ and is a decreasing when $t>0$.

Similarly, for a pair of decorated tableaux, we define the \textbf{non-intersecting line ensemble }$M_{\tilde{L},\tilde{R}}:\bZ\times[-\th,\th]\to\bZ$
of the pair $\left(\tilde{L},\tilde{R}\right)\in\cT_{n}$ by
\[
M_{\tilde{L},\tilde{R}}(i;t)=\begin{cases}
M_{\tilde{L}}(i;\th+t) & \ t\leq0\\
M_{\tilde{R}}(i;\th-t) & \ t\geq0
\end{cases}.
\]
Again, each of the lines are well defined and continuous at $t=0$ because $M_{\tilde{L}}(i;\th)=M_{\tilde{R}}(i;\th)$. See Figure \ref{fig:small_example} for an example of this line ensemble.
\end{defn}

\newcommand\rA{0.2}
\newcommand\rB{0.3}
\newcommand\rC{0.5}
\newcommand\rD{0.8}

\newcommand\lA{0.1}
\newcommand\lB{0.4}
\newcommand\lC{0.6}
\newcommand\lD{0.7}

\begin{figure}
	\begin{subfigure}[b]{.25\textwidth}
	\begin{center}
		\begin{tikzpicture}[scale = 3]
			\coordinate (BL) at (0,0);
			\coordinate (BR) at (1,0);
			\coordinate (TL) at (0,1);
			\coordinate (TR) at (1,1);

			\draw [black] (BL) -- (BR);
			\draw [black] (BL) -- (TL);			
			\draw [black] (TL) -- (TR);
			\draw [black] (BR) -- (TR);
			
			\coordinate (E) at (\rA, \lD);
			\coordinate (F) at (\rB, \lA);
			\coordinate (G) at (\rC, \lC);
			\coordinate (H) at (\rD, \lB); 	
		
			\fill (E) circle[radius=0.5pt];
			\fill (F) circle[radius=0.5pt];
			\fill (G) circle[radius=0.5pt];
			\fill (H) circle[radius=0.5pt];
			
			\foreach \x in {0,0.1,0.2,0.3,0.4,0.5,0.6,0.7,0.8,0.9,1.0}
   			    \draw (\x,0pt) -- (\x,-0.5pt);

			\foreach \x in {0.0,0.5,1.0}
   			    \draw (\x,0pt) -- (\x,-0.5pt)
   			     node[anchor=north] {$\x$};
   			     
   			     \foreach \y in {0,0.1,0.2,0.3,0.4,0.5,0.6,0.7,0.8,0.9,1.0}
   			    \draw (0pt,\y) -- (-0.5pt,\y);
   			    
   			    \foreach \y in {0.0,0.5,1.0}
   			    \draw (0pt,\y) -- (-0.5pt,\y)
   			     node[anchor=east] {$\y$};

		\end{tikzpicture}
		\end{center}
		\caption{Point configuration}
	\end{subfigure}
	\begin{subfigure}[b]{.33\textwidth}
		\begin{center}
		\ytableausetup{centertableaux,boxsize=2.5em} 
		\begin{ytableau} 
		\stackrel{1}{\lA} & \stackrel{2}{\lB} \\
		\stackrel{3}{\lC} \\
		\stackrel{4}{\lD} \\
		\end{ytableau}
		\ytableausetup{centertableaux,boxsize=2.5em} 
		\begin{ytableau} 
		\stackrel{1}{\rA} & \stackrel{3}{\rC} \\
		\stackrel{2}{\rB} \\
		\stackrel{4}{\rD} \\
		\end{ytableau}
		\vspace{14pt}
		\end{center}
		\caption{Pair of decorated Young tableaux}
	\end{subfigure}
	\begin{subfigure}[b]{.3\textwidth}
		\begin{center}
		\begin{tikzpicture}[xscale = 3,yscale=0.5]
			
			\foreach \y in {2,1,0,-1,-2,-3,-4}
   			    \draw[thin,densely dotted](1,\y) -- (-1,\y)
   			     node[anchor=east] {$\y$};
   			     
   			\coordinate (Laxis) at (-1,-4.5);
   			\coordinate (Raxis) at (1,-4.5);
			\coordinate (O) at (0,-4.5);
			\coordinate (vert) at (0,3);   			
   			\draw[black,<->] (Laxis) -- (Raxis); 
   			\draw[black,->] (O) -- (vert); 
   			
   			\foreach \x in {-0.9,-0.8,-0.7,-0.6,-0.5,-0.4,-0.3,-0.2,-0.1,0,0.9,0.8,0.7,0.6,0.5,0.4,0.3,0.2,0.1}
   			    \draw(\x,-4.5) -- (\x,-4.6);
   			    
   			\foreach \x in {-1.0,-0.5,0,0.5,1.0}
   			    \draw (\x,-4.5) -- (\x,-4.6)
   			     node[anchor=north] {$\x$};
   			     
   			\draw (0, -5.3) -- (0, -5.3) node[anchor=north] {t};

			\coordinate (LA) at (-1,-1);
			\coordinate (LB) at (-1,-2);
			\coordinate (LC) at (-1,-3);
			\coordinate (LD) at (-1,-4);
			\coordinate (LE) at (-1,-5);
			
			\coordinate (RA) at (1,-1);
			\coordinate (RB) at (1,-2);
			\coordinate (RC) at (1,-3);
			\coordinate (RD) at (1,-4);
						
			\draw [black,thick] (LD) -- (RD);
						
			\coordinate (JAa) at (-1.0+\lA,-1);
			\coordinate (JAb) at (-1.0+\lA,0);
			\coordinate (JAc) at (-1.0+\lB,0);
			\coordinate (JAd) at (-1.0+\lB,1);
			
			\coordinate (JAe) at (1.0-\rC,1);
			\coordinate (JAf) at (1.0-\rC,0);
			\coordinate (JAg) at (1.0-\rA,0);
			\coordinate (JAh) at (1.0-\rA,-1);

			\draw [black,thick] (LA) -- (JAa);
			\draw [black,thick] (JAa) -- (JAb);
			\draw [black,thick] (JAb) -- (JAc);
			\draw [black,thick] (JAc) -- (JAd);
			\draw [black,thick] (JAd) -- (JAe);
			\draw [black,thick] (JAe) -- (JAf);
			\draw [black,thick] (JAf) -- (JAg);
			\draw [black,thick] (JAg) -- (JAh);
			\draw [black,thick] (JAh) -- (RA);
						
			\coordinate (JBa) at (-1.0+\lC,-2);
			\coordinate (JBb) at (-1.0+\lC,-1);
			
			\coordinate (JBc) at (1.0-\rB,-1);
			\coordinate (JBd) at (1.0-\rB,-2);
			
			\draw [black,thick] (LB) -- (JBa);
			\draw [black,thick] (JBa) -- (JBb);
			\draw [black,thick] (JBb) -- (JBc);
			\draw [black,thick] (JBc) -- (JBd);
			\draw [black,thick] (JBd) -- (RB);
			
			\coordinate (JCa) at (-1.0+\lD,-3);
			\coordinate (JCb) at (-1.0+\lD,-2);
			
			\coordinate (JCc) at (1.0-\rD,-2);
			\coordinate (JCd) at (1.0-\rD,-3);
			
			\draw [black,thick] (LC) -- (JCa);
			\draw [black,thick] (JCa) -- (JCb);
			\draw [black,thick] (JCb) -- (JCc);
			\draw [black,thick] (JCc) -- (JCd);
			\draw [black,thick] (JCd) -- (RC);

		\end{tikzpicture}
		\caption{Non-intersecting line ensemble}
		
		\end{center}
	\end{subfigure}
	\caption{An example of the decorated Robinson-Schensted correspondence applied to a particular configuration from $C_n^{\th}$ when $\th=1.0$ and $n=4$. The lines of the associated non-intersecting line ensemble, $M_{\tilde{L},\tilde{R}}(i;t)$, are also plotted for $i=1,2,3,4$. In this example, the point configuration is $\binom{x_{1}\ x_{2}\ x_{3} \ x_{4}}{y_{1\ }y_{2}\ y_{3} \ y_{n}} = \binom{ \rA \ \rB \ \rC \ \rD}{\lD \ \lA \ \lC \ \lB}$.  The sorted $x$ and $y$ coordinates are respectively $(r_1,r_2,r_3,r_4)=(\rA,\rB,\rC,\rD)$ and $(\ell_1,\ell_2,\ell_3,\ell_4)=(\lA,\lB,\lC,\lD)$, and the associated permutation is $\si = \binom{1 \ 2 \ 3 \ 4}{4 \ 1 \ 3 \ 2}$.  Note that in this case the functions $M_{\tilde{L},\tilde{R}}(i,\cdot)\equiv -i$ are constant for $i \geq 4$; only the first three lines are non-constant. All three pictures contain exactly the same information because of the bijections between the three objects.}
\label{fig:small_example}
\end{figure}

\section{The Poissonized Robinson-Schensted Process}

\subsection{Definition}

\begin{defn}
Fix a parameter $\th \in \bR^+$. A rate 1 Poisson point process in $[0,\th]\times[0,\th]$ is a probability
measure on the set of configurations $\cC^{\th}$. By applying the
decorated RS correspondence this induces a probability measure on
$\cT^{\th}$. We will refer to the resulting random pair of tableaux
$\left(\tilde{L},\tilde{R}\right)\in\cT^{\th}$ as the \textbf{Poissonized RS tableaux}, we refer to the resulting random Young diagram process as the \textbf{Poissonized RS process} and the resulting random non-intersecting line ensemble as the \textbf{Poissonized RS line ensemble}. A realization of the Poissonized RS line ensemble for the case $\th = 40.0$ is displayed in Figure \ref{thepic}.

The main results of this article are to characterize the law of the Poissonized RS process. Both the Young diagram process and the non-crossing line ensemble of this object have natural descriptions. In this section we describe the laws of these. For the rest of the section, denote by $\left(\tilde{L},\tilde{R}\right)=\left(\left(L,(\ell_{1},\ld,\ell_{\abs{\sh L}})\right),\left(R,(r_{1},\ld,r_{\abs{\sh R}})\right)\right)$ a Poissonized RS random variable. The Young diagram process $\la_{\tilde{L},\tilde{R}}$ is a $\bY$ valued stochastic process,
and $M_{\tilde{L},\tilde{R}}$ is a random non-intersecting line
ensemble.
\end{defn}

\subsection{Law of the Young diagram process}
\begin{thm}
\label{SchurProcessThm}Fix $\th>0$ and times $t_{1}<\ld<t_{n}\in[-\th,0)$
and $s_{1}<\ld<s_{m}\in(0,\th]$. Suppose we are given an increasing
list of Young diagrams $\la^{(1)}\subset\ld\subset\la^{(n)}$, a decreasing
list of Young diagrams $\mu^{(1)}\supset\ld\supset\mu^{(m)}$ and
a Young diagram $\nu$ with $\nu\supset\la^{(n)}$ and $\nu\supset\mu^{(1)}$.
To simplify the presentation we will use the convention $\la^{(0)}=\emptyset,\la^{(n+1)}=\nu,\mu^{(0)}=\nu,\mu^{(m+1)}=\emptyset$
and $t_{0}=-\th,t_{n+1}=0,s_{0}=0,s_{m+1}=\th$. The Poissonized RS process $\la_{\tilde{L},\tilde{R}}$ has the following finite dimensional distribution:

\begin{eqnarray*}
 &  & \p\left( \bigcap_{i=1}^{n}\left\{ \la_{\tilde{L},\tilde{R}}(t_{i})=\la^{(i)}\right\} \cap \left\{ \la_{\tilde{L},\tilde{R}}(0)=\nu\right\} \cap \bigcap_{j=1}^{m} \left\{ \la_{\tilde{L},\tilde{R}}(s_{j})=\mu^{(j)}\right\} \right)\\
 & = & e^{-\th^{2}} \left( \prod_{i=0}^{n} \dim(\la^{(i+1)}/\la^{(i)}) \frac{(t_{i+1} - t_{i})^{ \abs{\la^{(i+1)}/\la^{(i)}} }}{\abs{\la^{(i+1)}/\la^{(i)}}!} \right) \cdot \left( \prod_{j=0}^{m} \dim(\mu^{(j)}/\mu^{(j+1)})\frac{(s_{j+1} - s_{j})^{\abs{\mu^{(j)}/\la^{(j+1)}}}}{\abs{\mu^{(j)}/\mu^{(j+1)}}!} \right),
\end{eqnarray*}
where $\dim(\la/\mu)$ is the number of standard Young tableau of skew shape $\la/\mu$.

\end{thm}

\begin{figure}
\begin{center}
\includegraphics[clip,width=6in]{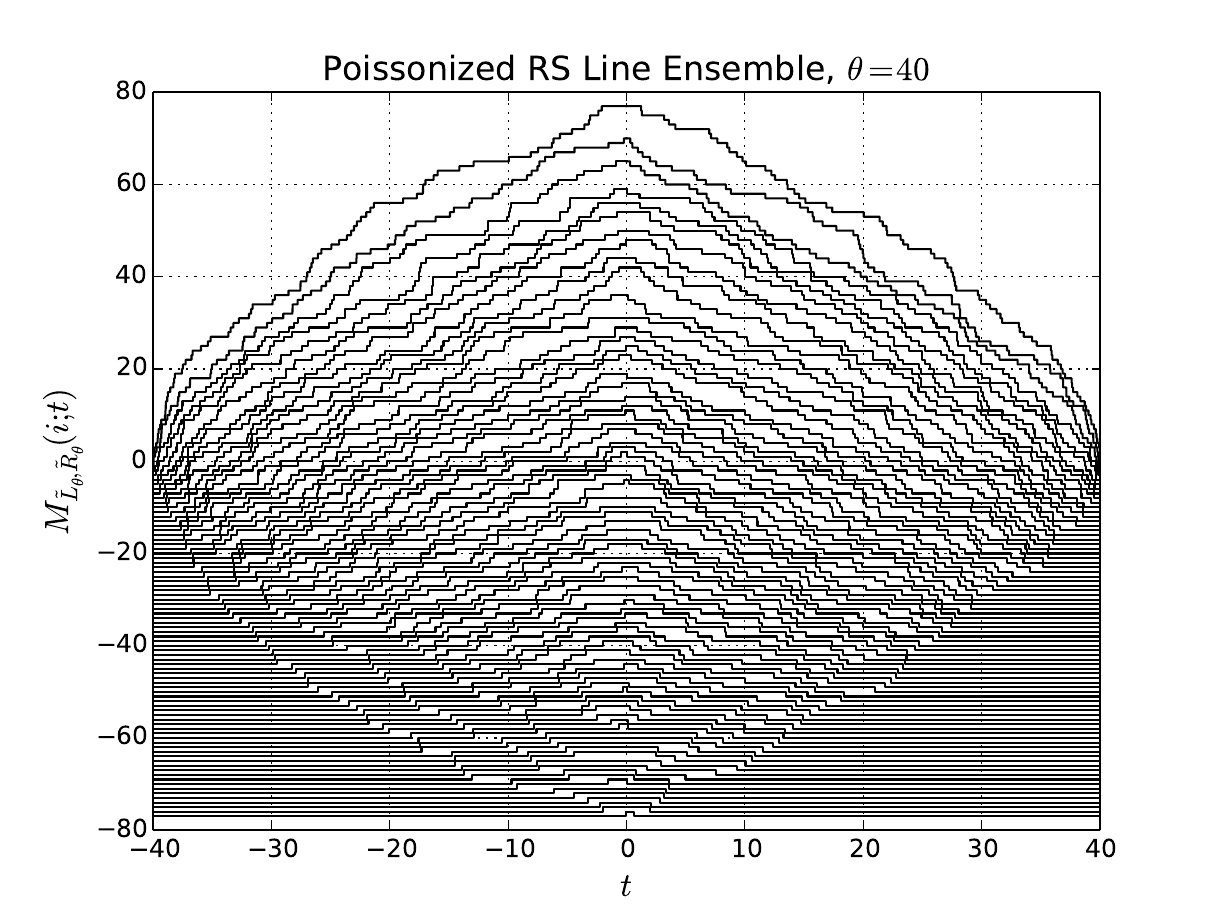}
\end{center}
\caption{A realization of the non-intersecting line ensemble for the Poissonized
RS process in the case $\th=40$ created from $\approx 1600$ points
in the plane. Only the lines that are non-constant are shown. This can be simulated efficiently because the tableaux are created by exactly one execution of the decorated Robinson-Schensted algorithm applied to a realization of a Poisson point process in the plane. }
\label{thepic}
\end{figure}

The proof of Theorem \ref{SchurProcessThm} is deferred to Subsection \ref{subsectionpf} and is proven using simple properties of the Robinson Schensted correspondence and probabilistic arguments. 

\begin{rem}
\label{ShurProcRem}
The conclusion of the theorem can be rewritten in a very algebraically satisfying way in terms of Schur functions specialized by the Plancherel specialization, $\rho(t)$ (see Subsection \ref{sec:notation} for our notations) Using the identity from Equation \ref{schurid}, the result of Theorem \ref{SchurProcessThm} can be rewritten as
\begin{eqnarray*}
 &  & \p\left(\bigcap_{i=1}^{n}\left\{ \la_{\tilde{L},\tilde{R}}(t_{i})=\la^{(i)}\right\} \cap \left\{ \la_{\tilde{L},\tilde{R}}(0)=\nu\right\} \cap \bigcap_{j=1}^{m}\left\{ \la_{\tilde{L},\tilde{R}}(s_{j})=\mu^{(j)}\right\} \right)\\
 & = & e^{-\th^{2}}\left(\prod_{i=0}^{n}s_{\la^{(i+1)}/\la^{(i)}}\left(\rho_{t_{i+1}-t_{i}}\right)\right)\cdot\left(\prod_{j=0}^{m}s_{\mu^{(j)}/\mu^{(j+1)}}\left(\rho_{s_{j+1}-s_{j}}\right)\right).
\end{eqnarray*}
In the literature (see for instance \cite{Okounkov01correlationfunction}
or \cite{1212.3351} for a survey) this type of distribution arising from specializations on sequence of Young diagrams is known as a Schur process. The particular Schur process
that appears here has a very simple ``staircase'' diagram, illustrated
here in the case $n=m=2$:

\begin{center}
\begin{tikzpicture}[scale=1.0] 
\node (A) at (0,0) {$\emptyset$}; 
\node (B) at (1,1) {$\lambda_{\tilde{L},\tilde{R}}(t_1)$}; 
\node (C) at (2,2) {$\lambda_{\tilde{L},\tilde{R}}(t_2)$};
\node (O) at (3,3) {$\lambda_{\tilde{L},\tilde{R}}(0)$}; 
\node (X) at (4,2) {$\lambda_{\tilde{L},\tilde{R}}(s_1)$}; 
\node (Y) at (5,1) {$\lambda_{\tilde{L},\tilde{R}}(s_2)$};
\node (Z) at (6,0) {$\emptyset$}; 
\path[->,font=\scriptsize] 
(A) edge node[left]{$\rho_{t_1 + \theta}$} (B) 
(B) edge node[left]{$\rho_{t_2 - t_1}$} (C)
(C) edge node[left]{$\rho_{0 - t_2}$} (O)
(O) edge node[right]{$\rho_{s_1 - 0}$} (X)
(X) edge node[right]{$\rho_{s_2 - s_1}$} (Y)
(Y) edge node[right]{$\rho_{\theta - s_2}$} (Z);
\end{tikzpicture}
\end{center}
In Section 4, we will further see that the Poissonized RS process $\la_{\tilde{L},\tilde{R}}$ is the same as a particular
instance model introduced in \cite{Borodin_stochasticdynamics}, which
is itself a special case of dynamics studied in \cite{Borodin04markovprocesses}.
\end{rem}
\begin{cor}
At any fixed time $t$, the Young diagram $\la_{\tilde{L},\tilde{R}}\left(t\right)$
has the law of the Poissonized Plancherel measure with parameter $\sqrt{\th\left(\th-\abs t\right)}$.\end{cor}
\begin{proof}
Suppose first that $t\leq0$ with the case $t\geq0$ being analogous.
By Theorem \ref{SchurProcessThm}, we have the two time probability distribution of
$\la_{\tilde{L},\tilde{R}}$ at time $t$ and time $0$ is
\[
\p\left(\left\{ \la_{\tilde{L},\tilde{R}}\left(t\right)=\la\right\} \cap\left\{ \la_{\tilde{L},\tilde{R}}(0)=\nu\right\} \right)=e^{-\th^{2}}s_{\la}\left(\rho_{t+\th}\right)s_{\nu/\la}\left(\rho_{0-t}\right)s_{\nu}\left(\rho_{\th-0}\right).
\]
Summing over $\nu\in\bY$ and employing the Cauchy identity $\sum_{\mu}s_{\mu/\la}(\rho)s_{\mu}(\rho')=H(\rho;\rho')s_{\la}(\rho')$, and using $H(\rho_{a};\rho_{b})=\exp\left(ab\right)$ for the exponential
specialization, we have
\begin{eqnarray*}
\p\left(\la_{\tilde{L},\tilde{R}}\left(t\right)=\la\right) & = & e^{-\th^{2}}s_{\la}\left(\rho_{t+\th}\right)\sum_{\nu\in\bY}s_{\nu/\la}\left(\rho_{0-t}\right)s_{\nu}\left(\rho_{\th-0}\right)\\
 & = & e^{-\th^{2}}H(\rho_{-t};\rho_{\th})s_{\la}\left(\rho_{t+\th}\right)s_{\la}\left(\rho_{\th}\right)\\
 & = & e^{-\th^{2}}e^{\abs t\th}s_{\la}\left(\rho_{\th-\abs t}\right)s_{\la}\left(\rho_{\th}\right)\\
 & = & e^{-\left(\sqrt{\th(\th-\abs t}\right)^{2}}\left(\frac{\dim\left(\la\right)\left(\sqrt{\th(\th-\abs t)}\right)^{\abs{\la}}}{\abs{\la}!}\right)^{2},
\end{eqnarray*}
as desired.
\end{proof}

The Poissonized Plancherel measure and its asymptotics are well studied,
see for example \cite{Borodin00asymptoticsof} or \cite{Johansson01discreteorthogonal}.
The analysis lets us see that, for any fixed $t$, the points of the line ensemble $M_{\tilde{L},\tilde{R}}\left(\cdot;t\right)$
form a determinantal point process whose kernel is the discrete Bessel
kernel. We can also use these results to write some asymptotics for
the Poissonized RS line ensemble,
for instance the following:
\begin{cor}
Let $\left(\tilde{L}_{\th},\tilde{R}_{\th}\right)$ be the Poissonized
RS tableaux of parameter $\th$. For fixed $\tau\in(-1,1)$,
the \uline{top line} $M_{\tilde{L}_{\th},\tilde{R}_{\th}}(1;\cdot)$
of the line ensemble at some fixed time $t$ satisfies the following
law of large numbers type behavior:

\[
\lim_{\th\to\infty}\frac{M_{\tilde{L}_{\th},\tilde{R}_{\th}}(1;\tau\th)}{\th}=2\sqrt{1-\abs{\tau}}\text{ a.s.}
\]
The fluctuations are of the Tracy-Widom type:
\[
\lim_{\th\to\infty}\p\left(\frac{M_{\tilde{L}_{\th},\tilde{R}_{\th}}(1;\tau\th)-2\th\sqrt{1-\abs{\tau}}}{\th^{1/3}(1-\abs{\tau})^{1/6}}\leq s\right)=F(s),
\]
where $F(s)$ is the GUE Tracy Widom distribution. 
\end{cor}

\subsection{The non-intersecting line ensemble}
\begin{defn}
Fix a parameter $\th>0$ and an initial location $x\in\bZ$. Let $P_{L}:\left[0,\th\right]\to\bZ$
and $P_{R}:\left[0,\th\right]\to\bZ$ be two independent rate 1 Poisson jump
processes with initial condition $P_{L}(0)=P_{R}(0)=x$. Define $P:[-\th,\th]\to\bZ$
by
\[
P(t)=\begin{cases}
P_{L}(\th+t) & t<0\\
P_{R}(\th-t) & t\geq0
\end{cases}.
\]

A \textbf{Poisson arch} on $[-\th,\th]$ with initial location $x$
is the stochastic process $\left\{ A(t)\right\} _{t\in[-\th,\th]}$
whose probability distribution is the conditional probability distribution
of the process $\left\{ P(t)\right\} _{t\in[-\th,\th]}$ conditioned
on the event that $\left\{ P_{L}(\th)=P_{R}(\th)\right\} $ . This
has $A(-\th)=A(\th)=x$ and the conditioning ensures that $A(t)$
is actually continuous at $t=0$.
\end{defn}

The Poissonized RS line ensemble, $M_{\tilde{L},\tilde{R}}(\cdot ; \cdot)$ has
a simple description in terms of Poisson arches which are conditioned not to intersect:
\begin{thm}
\label{ArchesThm}Fix $\th>0$ and times $-\th<t_{1}<\ld<t_{n}<\th$.
For any $N\in\bN$, consider a non-intersecting line ensemble $A:\left\{ 1,2,\ld,N\right\} \times[-\th,\th]\to\bZ$,
so that $\left\{ A(i;\cdot)\right\} _{i=1}^{N}$ is a collection of
$N$ Poisson arches on $[-\th,\th]$ with the initial condition $A_{i}(-\th)=A_{i}(\th)=-i$
which are conditioned not to intersect i.e. $A(i;t)<A(j;t)$ for all
$i>j$. Then the joint probability distributions of the line ensemble
$A$ has the same conditional distribution as top $N$ lines of the
non-intersecting line ensemble $M_{\tilde{L},\tilde{R}}$, conditioned
on the event that all of the other lines $M_{\tilde{L},\tilde{R}}\left(k;\cdot\right)$
for $k>N$ do not move at all. To be precise, for fixed target points
$\left\{ x_{i,j}\right\} _{1\leq i\leq n,1\leq j\leq N}$ we have:
\[
\p\left(\bigcap_{i=1}^{n}\left(\bigcap_{j=1}^{N}\left\{ A(j;t_{i})=x_{i,j}\right\} \right)\right)=\p\left(\bigcap_{i=1}^{n}\left(\bigcap_{j=1}^{N}\left\{ M_{\tilde{L},\tilde{R}}(j;t_{i})=x_{i,j}\right\} \right)\given{M_{\tilde{L},\tilde{R}}(k;\cdot)\equiv -k\ \forall k>N}\right).
\]
\end{thm}

The proof of this goes through the Karlin-MacGregor/Lindstr\"{o}m–Gessel–Viennot theorem and is deferred to Section \ref{subsection_ArchesThm}

\subsection{Proof of Theorem \ref{SchurProcessThm}} \label{subsectionpf}

We prove this theorem by splitting it into several lemmas. The idea
behind these lemmas is to exploit the fact that decorations and the tableaux that make up the pair of decorated tableaux of the Poissonized RS process are conditionally independent when conditioned on certain carefully chosen events.




\newcommand\Cti{ C_{t_i}\big( \la^{(i)} \big) }
\newcommand\Sti{ S_{t_i}\big( |\la^{(i)}| \big) }
\newcommand\Csi{ C_{s_i}\big( \mu^{(i)} \big) }
\newcommand\Ssi{ S_{s_i}\big( |\mu^{(i)}| \big) }
\newcommand\Co{ C_{0}\big( \nu \big) }


\begin{defn}
\label{def:shorthands}
For any $-\th < t<\th$, Young diagram $\la$, and any $k\in\bN$ define the shorthand notations:
\begin{eqnarray*}
C_{t}\big( \la \big) &:=&  \left\{ \la_{\tilde{L},\tilde{R}}(t)=\la \right\}, \\ 
S_{t}\big( k \big) &:=&  \left\{ \abs{\la_{\tilde{L},\tilde{R}}(t)}= k \right\}.  
\end{eqnarray*}
With this notation, Theorem \ref{SchurProcessThm} is an explicit formula for the probability of the event  $\bigcap_{i=1}^{n} \Cti \cap \Co \cap \bigcap_{j=1}^{m} \Csi$. The events $\Sti$ and $\Ssi$ will also appear in our arguments below.

$C_{t}(\la)$ is the event that the Young diagram process at time $t$ is exactly equal to $\la$, while $S_{t}(|\la|) \subset C_{t}(\la)$ is the event that the Young diagram process at time $t$ has the same \uline{size} as $\la$ (but is possibly a different shape).  
\end{defn}

\begin{lem}
\label{conditioninglemma}Let $N=\abs{\la_{\tilde{L},\tilde{R}}(0)}$.
Then $N$ has the distribution of a Poisson random variable, $N\sim Poisson(\th^{2})$.
Moreover, conditioned on the event $\left\{ N=n\right\} $, the decorations
$\left(\ell_{1},\ld,\ell_{n}\right)$ and $\left(r_{1},\ld,r_{n}\right)$
of $\left(\tilde{L},\tilde{R}\right)$ are independent. Still conditioned
on $\{N=n\}$, the permutation $\si\in S_{n}$ associated with $\left(L,R\right)$
via the RS correspondence is uniformly distributed in $S_{n}$ and
is independent of both sets of decorations. \end{lem}
\begin{proof}
From the construction of the Poissonized RS process, $\abs{\la_{\tilde{L},\tilde{R}}(0)}$
is the number of points in the square $[0,\th]\times[0,\th]$ of a
rate 1 Poisson point process in the plane. Hence $N\sim Poisson(\th^{2})$
is clear. Conditioned on $\left\{ N=n\right\} $ the points of the
rate 1 Poisson process in question are uniformly distributed in the
square $[0,\th]\times[0,\th]$ (This is a general property of Poisson
point processes). Consequently, each point has an $x$-coordinate
and $y$-coordinate which are uniformly distributed in $[0,\th]$
and are independent of all the other coordinates. Hence, since the
decorations $\left(\ell_{1},\ld,\ell_{n}\right),\left(r_{1},\ld,r_{n}\right)$
consist of the sorted $y-$coordinates and $x$-coordinates respectively,
these decorations are independent of each other. (More specifically:
they have the distribution of the order statistics for a sample of
$n$ uniformly distributed points in $[0,\th]$.)

To see that the permutation $\si\in S_{n}$ associated with these points
is uniformly distributed in $S_{n}$ first notice that this is the
same as the permutation associated with the points of the Poisson
Point Process. Then notice that if $\left(a_{1},\ld,a_{n}\right),\left(b_{1},\ld,b_{n}\right)$
are independent drawings of the order statistics for a sample of $n$
uniformly distributed points in $[0,\th]$ and $\pi\in S_{n}$ is
drawn uniformly at random and independently of everything else, then
the points $\left\{ \left(a_{i},b_{\pi(i)}\right)\right\} _{i=1}^{n}$
is a sample of $n$ points chosen uniformly from $[0,\th]\times[0,\th]$.
This construction of the $n$ uniform points shows that the permutation
$\si$ is uniformly distributed and independent of both the $x$ and
$y$ coordinates.\end{proof}
\begin{cor}
\label{YTcor}
Recall the shorthand notations from Definition \ref{def:shorthands}. For any Young diagram $\nu$, we have
\[
\p\left( \Co \right)=\left(\frac{\dim(\nu)^{2}}{\abs{\nu}!}\right)\p\left(N=\abs{\nu}\right).
\]
\end{cor}
\begin{proof}
By construction, the Young diagram $\la_{\tilde{L},\tilde{R}}(0)$
is the common shape of the tableaux $(L,R)$. Hence $\Co = \left\{ \la_{\tilde{L},\tilde{R}}(0)=\nu\right\} =\left\{ N=\abs{\nu}\right\} \cap\left\{ \sh(L)=\nu\right\} $.
Notice that $\sh(L)$ depends only on the associated permutation $\si$,
whose conditional distribution is known here to be uniform in $S_{n}$
by Lemma \ref{conditioninglemma}. Since the RS correspondence is a
bijection, we have only to count the number of pairs of tableaux of
shape $\nu$. Hence:
\begin{eqnarray*}
\p\left( \Co \right) & = & \p\left(\sh(L)=\nu\given{N=\abs{\nu}}\right)\p\left(N=\abs{\nu}\right)\\
 & = & \left(\frac{\dim(\nu)^{2}}{\abs{\nu}!}\right)\p\left(N=\abs{\nu}\right).
\end{eqnarray*}
\end{proof}
\begin{lem}
\label{sizelemma} Recall the shorthand notations from Definition \ref{def:shorthands}. Consider the law of the process conditioned on the
event $C_0(\nu)$. The conditional
probability that $\la_{\tilde{L},\tilde{R}}$ has the correct \uline{sizes}
at times $t_{1}<.\ld<t_{n}\in[-\th,0]$ is given by
\begin{eqnarray*}
\p\left(\bigcap_{i=1}^{n} S_{t_i}\big( |\la^{(i)}| \big) \given{ C_0\big(\nu\big) }\right) & = & \frac{\prod_{i=0}^{n}P_{\th\left(t_{i+1}-t_{i}\right)}\left(\abs{\la^{(i+1)}}-\abs{\la^{(i)}}\right)}{\p\left(N=\abs{\nu}\right)},
\end{eqnarray*}
where $P_{r}(k)=e^{-r}\frac{r^{k}}{k!}$ is the Poisson probability
mass function.  An analogous formula holds for $\p\left(\bigcap_{i=1}^{m} \Ssi \given{\Co} \right)$. Moreover, we have the following type of conditional independence for the sizes at
times $-\th<t_{1}<.\ld<t_{n}<0$ and at times $0<s_{1}<.\ld<s_{m}<\th$:
$$
\p\left( \bigcap_{i=1}^{n} \Sti \cap \bigcap_{i=1}^{m} \Ssi \given{\Co} \right) = \p\left( \bigcap_{i=1}^{n} \Sti \given{\Co} \right)\cdot \p\left( \bigcap_{i=1}^{m} \Ssi \given{\Co} \right).
$$

\end{lem}
\begin{proof}
As in the previous lemma, we have $\Co =\left\{ N=\abs{\nu}\right\} \cap\left\{ \sh(L)=\nu\right\} $. Then consider:

\begin{eqnarray*}
 \p\left(\bigcap_{i=1}^{n}\Sti \cap \Co \right) & = & \p\left(\bigcap_{i=1}^{n} \Sti \cap \left\{ N=\abs{\nu}\right\} \cap\left\{ \sh(L)=\nu\right\} \right)\\
 & = & \p\left(\bigcap_{i=1}^{n} \Sti \cap\left\{ \sh(L)=\nu\right\} \given{N=\abs{\nu}}\right)\p\left(N=\abs{\nu}\right).
\end{eqnarray*}
But now, when conditioned on $\left\{N=\abs{\nu}\right\}$, the event $\bigcap_{i=1}^{n}\Sti$ and the event $\left\{ \sh(L)=\nu\right\} $ are independent. The former event depends only on the decorations $\left(\ell_{1},\ld,\ell_{n}\right)$ by the
definition of $\la_{\tilde{L},\tilde{R}}$ and since $t_{i}<0$, 
while the latter event depends only on the associated permutation $\si$,
and these are conditionally independent by Lemma \ref{conditioninglemma}.
Hence:

\begin{eqnarray*}
 &  & \p\left(\bigcap_{i=1}^{n} \Sti \cap\left\{ \sh(L)=\nu\right\} \given{N=\abs{\nu}}\right)\p\left(N=\abs{\nu}\right)\\
 & = & \p\left(\bigcap_{i=1}^{n} \Sti \given{N=\abs{\nu}}\right)\p\left(\sh(L)=\nu\given{N=\abs{\nu}}\right)\p\left(N=\abs{\nu}\right)\\
 & = & \p\left(\bigcap_{i=1}^{n} \Sti \given{N=\abs{\nu}}\right)\p\Big( \Co \Big).
\end{eqnarray*}
Putting the above displays together, we have:
\[
\p\left(\bigcap_{i=1}^{n} \Sti \given{\Co}\right)=\p\left(\bigcap_{i=1}^{n} \Sti \given{N=\abs{\nu}}\right).
\]
Now from the definition of the Young diagram process, we have for
$-\th<t<0$, that $\abs{\la_{\tilde{L},\tilde{R}}(t)}=\abs{\left\{ i:\ \ell_{i}<t+\th\right\} }$
and we see that the event $\bigcap_{i=1}^{n}\Sti \cap\left\{ N=\abs{\nu}\right\} $
depends only on counting the number of decorations from $\left(\ell_{1},\ld,\ell_{\abs{\nu}}\right)$
in the appropriate regions:

\begin{eqnarray*}
\bigcap_{i=1}^{n} \Sti \cap\left\{ N=\abs{\nu}\right\}  & = & \bigcap_{i=0}^{n+1}\left\{ \abs{\la_{\tilde{L},\tilde{R}}(t_{i})}=\abs{\la^{(i)}}\right\} \\
 & = & \bigcap_{i=0}^{n+1}\left\{ \left|\left\{ j\ :\ t_{i}+\th<\ell_{j}<t_{i+1}+\th\right\} \right|=\abs{\la^{(i+1)}}-\abs{\la^{(i)}}\right\}.
\end{eqnarray*}
Finally, from the construction, we notice that the random variable
$\abs{\left\{ j:t_{i}+\th<\ell_{j}<t_{i+1}+\th\right\} }$ counts
the number of points of the Poisson point process in the region $[t_{i}+\th,t_{i+1}+\th]\times[0,\th]$.
Consequently these random variables are independent for different
values of $i$ and are distributed according to 
\[
\abs{\left\{ j:t_{i}+\th<\ell_{j}<t_{i+1}+\th\right\} }\sim Poisson\left(\th\left(t_{i+1}-t_{i}\right)\right).
\]
This observation, together with the preceding display, gives the desired first result of the lemma.

To see the second result about the conditional independence at times
$-\th<t_{1}<.\ld<t_{n}<0$ and at times $0<s_{1}<.\ld<s_{m}<\th$,
we repeat the arguments above and notice that times $-\th<t_{1}<.\ld<t_{n}<0$
depend only on the decorations $(\ell_{1},\ld,\ell_{n})$ (because
of $\abs{\la_{\tilde{L},\tilde{R}}(t)}=\abs{\left\{ i:\ \ell_{i}<t+\th\right\} }$)
while the times $0<s_{1}<\ld<s_{m}<\th$ depend only on the decorations
$(r_{1},\ld,r_{m})$ (because of $\abs{\la_{\tilde{L},\tilde{R}}(s)}=\abs{\left\{ i:\ r_{i}<\th-s\right\} }$).
These decorations are conditionally independent when conditioned on
$\{N=\abs{\nu}\}$ by Lemma \ref{conditioninglemma} and the desired
independence result follows.\end{proof}
\begin{lem}
\label{shapelem}
Recall the shorthand notations from Definition \ref{def:shorthands}. We
have
\[
\p\left(\bigcap_{i=0}^{n+1}\Cti \given{\bigcap_{i=0}^{n+1} \Sti \cap \Co}\right)=\frac{\prod_{i=1}^{n+1}\dim\left(\la^{(i+1)}/\la^{(i)}\right)}{\dim(\nu)}.
\]

An analogous formula holds for $\p\left(\bigcap_{i=0}^{m+1} \Csi \given{\bigcap_{i=0}^{m+1} \Ssi \cap \Co}\right)$. Moreover, we have the following type of conditional independence at
times $-\th<t_{1}<.\ld<t_{n}<0$ and at times $0<s_{1}<.\ld<s_{m}<\th$:
\begin{eqnarray*}
 &  & \p\left(\bigcap_{i=0}^{n+1} \Cti  \cap \bigcap_{i=0}^{m+1} \Csi \given{\bigcap_{i=0}^{n+1} \Sti \cap\bigcap_{i=0}^{m+1} \Ssi \cap \Co }\right)\\
 & = & \p\left(\bigcap_{i=0}^{n+1}\Cti \given{\bigcap_{i=0}^{n+1} \Sti \cap \Co}\right)\cdot\p\left(\bigcap_{i=0}^{m+1} \Csi \given{\bigcap_{i=0}^{m+1} \Ssi \cap \Co}\right).
\end{eqnarray*}
\end{lem}
\begin{proof}
For a standard Young tableau $T$, and $a<b\in\bN$ we will denote by
$\sh(T_{a,b})$ the skew Young diagram which consists of the boxes
of $T$ which are labeled with an number $i$ so that $a\leq i\leq b$
and the empty Young diagram in the case $b<a$. With this notation,
we now notice that for $t<0$, that $C_{t}(\la) = \left\{ \la_{\tilde{L},\tilde{R}}(t)=\la\right\} = \left\{ \sh(L_{1,\abs{\la_{\tilde{L},\tilde{R}}(t)}})=\la\right\} $.
By the same token we have:
\[
\bigcap_{i=1}^{n+1}\Cti  = \bigcap_{i=0}^{n+1}\left\{ \sh(L_{1,\abs{\la_{\tilde{L},\tilde{R}}(t_{i})}})=\la^{(i)}\right\} = \bigcap_{i=0}^{n}\left\{ \sh(L_{\abs{\la_{\tilde{L},\tilde{R}}(t_{i})}+1,\abs{\la_{\tilde{L},\tilde{R}}(t_{i+1})}})=\la^{(i+1)}/\la^{(i)}\right\}. 
\]
Hence:
\begin{eqnarray*}
 &  & \p\left(\bigcap_{i=0}^{n+1}\Cti \given{\bigcap_{i=0}^{n+1}\Sti \cap \Co}\right)\\
 & = & \p\left(\bigcap_{i=0}^{n+1}\left\{ \sh(L_{\abs{\la_{\tilde{L},\tilde{R}}(t_{i})}+1,\abs{\la_{\tilde{L},\tilde{R}}(t_{i+1})}})=\la^{(i+1)}/\la^{(i)}\right\} \given{\bigcap_{i=0}^{n}\left\{ \abs{\la_{\tilde{L},\tilde{R}}(t_{i})}=\abs{\la^{(i)}}\right\} \cap \Co}\right)\\
 & = & \p\left(\bigcap_{i=0}^{n+1}\left\{ \sh(L_{\abs{\la^{(i)}}+1,\abs{\la^{(i+1)}}})=\la^{(i+1)}/\la^{(i)}\right\} \given{\bigcap_{i=0}^{n}\left\{ \abs{\la_{\tilde{L},\tilde{R}}(t_{i})}=\abs{\la^{(i)}}\right\} \cap \Co}\right).
\end{eqnarray*}
We now notice that this event depends only on the Young
tableau $L$, which is entirely determined by the associated permutation
$\si$ (via the RS algorithm). Consequently, the conditioning on $\bigcap_{i=0}^{n}\left\{ \abs{\la_{\tilde{L},\tilde{R}}(t_{i})}=\abs{\la^{(i)}}\right\} $,
which depends only on the decorations $\left(\ell_{1},\ld,\ell_{\abs{\la}}\right)$,
has no effect here since $\si$ and these decorations are conditionally
independent by Lemma \ref{conditioninglemma}. Removing this conditioning
on $\bigcap_{i=0}^{n}\left\{ \abs{\la_{\tilde{L},\tilde{R}}(t_{i})}=\abs{\la^{(i)}}\right\} $,
we remain with:
\[
\p\left(\bigcap_{i=0}^{n+1}\Cti \given{\bigcap_{i=0}^{n+1} \Sti \cap \Co }\right) = \p\left(\bigcap_{i=0}^{n+1}\left\{ \sh(L_{\abs{\la^{(i)}}+1,\abs{\la^{(i+1)}}})=\la^{(i+1)}/\la^{(i)}\right\} \given{\Co}\right).
\]
With this conditioning, since $\si$ is uniformly distributed by Lemma \ref{conditioninglemma},
and because the RS algorithm is a bijection, the Young tableau $L$
is uniformly distributed among the set of all Young tableau of shape
$\nu$. Hence it suffices to count the number of tableau of shape
$\nu$ with the correct intermediate shapes i.e. the tableaux $L$
that have $\sh(L_{\abs{\la^{(i)}}+1,\abs{\la^{(i+1)}}})=\la^{(i+1)}/\la^{(i)}$
for each $i$. Since each $L_{\abs{\la^{(i)}}+1,\abs{\la^{(i+1)}}}$must
itself be a standard Young tableau of skew shape $\la^{(i+1)}/\la^{(i)}$,
this is:
\begin{eqnarray*}
 &  & \left|\left\{ \text{S.Y.T. }L\ :\sh(L)=\nu,\ \sh\left(L_{\abs{\la^{(i)}}+1,\abs{\la^{(i+1)}}}\right)=\la^{(i+1)}/\la^{(i)}\ \forall1\leq i\leq n\right\} \right|\\
 & = & \dim(\la_{1}/\emptyset)\dim\left(\la_{2}/\la_{1}\right)\ld\dim\left(\la_{n}/\la_{n-1}\right)\dim\left(\nu/\la_{n}\right).
\end{eqnarray*}
 Dividing by $\dim(\nu),$ the total number of tableaux of shape $\nu$, gives the desired probability and completes the first result of
the lemma.

To see the second result about the conditional independence at times
$-\th<t_{1}<.\ld<t_{n}<0$ and at times $0<s_{1}<.\ld<s_{m}<\th$,
we repeat arguments analogous to the above to see that
\[
\p\left(\bigcap_{i=0}^{n+1}\Cti \cap \bigcap_{i=0}^{m+1} \Csi \given{\bigcap_{i=0}^{n+1} \Sti \cap \bigcap_{i=0}^{m+1} \Ssi \cap \Co}\right)  = \p\left(A_{L}\cap A_{R}\given{\Co}\right),
\]
where
\begin{eqnarray*}
  A_{L} &:=& \bigcap_{i=0}^{n+1}\left\{ \sh(L_{\abs{\la^{(i)}}+1,\abs{\la^{(i+1)}}})=\la^{(i+1)}/\la^{(i)}\right\} \\
  A_{R} &:=& \bigcap_{i=0}^{m+1}\left\{ \sh(R_{\abs{\mu^{(i)}}+1,\abs{\mu^{(i+1)}}})=\mu^{(i+1)}/\mu^{(i)}\right\}. 
\end{eqnarray*}
Now the event $A_{L}$ depends only on the left tableau $L$ while
the event $A_{R}$ depends only on the right tableau $R$. By Lemma \ref{conditioninglemma} along with the fact that the RS correspondence is a bijection, we know that under the conditioning on the event $\Co$ that the tableaux $L$ and $R$ are independent of each other (both are uniformly distributed among the set of Young tableau of shape $\nu$). Thus the events $A_{L}$
and $A_{R}$ are conditionally independent when conditioned on $\Co$, yielding the desired independence result.
\end{proof}

\begin{proof}
(Of Theorem \ref{SchurProcessThm}). The proof goes by carefully deconstructing the desired probability and using the conditional independence results from Lemma \ref{sizelemma} and Lemma \ref{shapelem}
until we reach an explicit
formula. Recall the shorthand notations from Definition \ref{def:shorthands}.  We have:
\begin{eqnarray*}
 &  & \p\left(\bigcap_{i=1}^{n}\Cti \cap \Co \cap \bigcap_{i=1}^{m} \Csi \right)\\
 & = & \p\left(\bigcap_{i=1}^{n}\Cti \cap \bigcap_{i=1}^{m} \Csi \given{\Co}\right)\p\left(\Co\right)\\
 & = & \p\left(\bigcap_{i=0}^{n} \Cti \cap \bigcap_{i=0}^{m} \Csi \given{\bigcap_{i=0}^{n+1} \Sti \cap \bigcap_{i=0}^{m+1} \Ssi \cap \Co}\right)\cdot\\
 &  & \cdot\p\left(\bigcap_{i=0}^{n} \Sti \cap \bigcap_{i=0}^{m} \Ssi \given{\Co}\right)\cdot\p\left(\Co\right)\\
 & = & \p\left(\bigcap_{i=0}^{n}\Cti \given{\bigcap_{i=0}^{n} \Sti \cap \Co}\right)\cdot \p\left(\bigcap_{i=0}^{m} \Csi \given{\bigcap_{i=0}^{m} \Ssi \cap \Co}\right)\cdot\\
 &  & \cdot\p\left(\bigcap_{i=0}^{n} \Sti \given{\Co}\right)\p\left(\bigcap_{i=0}^{m} \Ssi \given{\Co}\right)\cdot\p\left(\Co\right)\\
 & = & \left(\frac{\prod_{i=0}^{n}\dim\left(\la^{(i+1)}/\la^{(i)}\right)}{\dim(\nu)}\right)\cdot\left(\frac{\prod_{i=0}^{m}\dim\left(\mu^{(i)}/\mu^{(i+1)}\right)}{\dim(\nu)}\right)\cdot\\
 &  & \cdot\left(\frac{\prod_{i=0}^{n}P_{\th\left(t_{i+1}-t_{i}\right)}\left(\abs{\la^{(i+1)}}-\abs{\la^{(i)}}\right)}{\p\left(N=\abs{\nu}\right)}\right)\cdot\left(\frac{\prod_{i=0}^{m}P_{\th\left(s_{i+1}-s_{i}\right)}\left(\abs{\la^{(i+1)}}-\abs{\la^{(i)}}\right)}{\p\left(N=\abs{\nu}\right)}\right)\cdot\left(\left(\frac{\dim(\nu)^{2}}{\abs{\nu}!}\right)\p\left(N=\abs{\nu}\right)\right).
\end{eqnarray*}
We now use $\p(N=\abs{\nu})=P_{\th^{2}}(\abs{\nu})=e^{-\th^{2}}\frac{\th^{2\abs{\nu}}}{\abs{\nu}!}$ to simplify the result. The desired result follows after simplifying the product using the Poisson probability mass formula.
\end{proof}

\subsection{Proof of Theorem \ref{ArchesThm}}
\label{subsection_ArchesThm}
\begin{proof}
(Of Theorem \ref{ArchesThm}) The proof will proceed as follows: First, by an application of the \newline Karlin-MacGregor/Lindstr\"{o}m–Gessel–Viennot theorem and the Jacobi-Trudi identity for Schur functions to compute
the distribution of the Poisson arches in terms of Schur functions.
Then, by Theorem \ref{SchurProcessThm}, the right hand side is computed to be the same
expression. 

For convenience of notation, divide the times into two
parts, times $-\th<t_{1}<\ld<t_{n}<0$ and $0<s_{1}<\ld<s_{m}<\th$,
and put $t_{0}=-\th,t_{n+1}=0,s_{0}=0,s_{m+1}=\th$. Set target points
$\left\{ x_{i,j}\right\} _{1\leq i\leq n,1\leq j\leq N}$ , $\left\{ z_{j}\right\} _{1\leq j\leq N}$and
$\left\{ y_{i,j}\right\} _{1\leq i\leq m,1\leq j\leq N}$ and consider
the event $\left\{ \bigcap_{i=1}^{n}\left(\bigcap_{j=1}^{N}\left\{ A(j;t_{i})=x_{i,j}\right\} \right)\cap\bigcap_{i=1}^{m}\left(\bigcap_{j=1}^{N}\left\{ A(j;s_{i})=y_{i,j}\right\} \right)\right\} $.
To each of the fixed time ``slices'', we Young diagram $\left\{ \la^{(i)}\right\} _{1\leq i\leq n},\left\{ \mu^{(i)}\right\} _{1\leq i\leq m}$
and $\nu$ by prescribing the length of the rows: 
\[
\la_{k}^{(i)}=\begin{cases}
x_{i,k}+k & 1\leq k\leq N\\
0 & k>N
\end{cases},\ \nu_{k}=\begin{cases}
z_{k}+k & 1\leq k\leq N\\
0 & k>N
\end{cases},\ \mu_{k}^{(i)}=\begin{cases}
y_{i,k}+k & 1\leq k\leq N\\
0 & k>N
\end{cases}.
\]
Reuse the same conventions as from Theorem \ref{SchurProcessThm}, $\la^{(0)}=\emptyset,\la^{(n+1)}=\nu,\mu^{(0)}=\nu,\mu^{(m+1)}=\emptyset$
and $t_{0}=-\th,t_{n+1}=0,s_{0}=0,s_{m+1}=\th$. Notice that by the
definitions, $\la^{(i)}$ and $\mu^{(j)}$ are always Young diagrams
with at most $N$ non-empty rows. Moreover, the admissible target
points are exactly in bijection with the space $\bY(N)$ of Young
diagrams with at most $N$ non-empty rows.

By application of the Karlin-MacGregor theorem \cite{karlin1959} / Lindstr\"{o}m–Gessel–Viennot theorem \cite{GesselViennot}  , for the law of non-intersecting random walks, we have that: 
\begin{eqnarray*}
 &  & \p\left(\bigcap_{i=1}^{n}\left(\bigcap_{j=1}^{N}\left\{ A(j;t_{i})=x_{i,j}\right\} \right)\cap\bigcap_{i=1}^{m}\left(\bigcap_{j=1}^{N}\left\{ A(j;s_{i})=y_{i,j}\right\} \right)\right)\\
 & = & Z_{t_{1},\ld t_{n}s_{1},\ld s_{m}}^{-1}\prod_{i=1}^{n+1}\det\left(W_{t_{i}-t_{i-1}}^{+}(x_{i-1,a},x_{i,b})\right)_{1\leq a,b\leq N}\cdot\prod_{i=1}^{m+1}\det\left(W_{s_{i}-s_{i-1}}^{-}(y_{i-1,a},y_{i,b})\right)_{1\leq a,b\leq N}.
\end{eqnarray*}
Here the weights $W_{t}^{+}$and $W_{s}^{-}$ are Poisson \uline{weights}
for an increasing/decreasing Poisson process:
\begin{eqnarray*}
W_{t}^{+}\left(x,y\right) & = & \frac{t^{(y-x)}}{(y-x)!}\one_{\left\{ y>x\right\} }\ , \ W_{t}^{-}=\frac{t^{(x-y)}}{(x-y)!}\one_{\left\{ x>y\right\} }.
\end{eqnarray*}

(We can safely ignore the factor of $e^{-t}$ that appears in the
transition probabilities as long as we are consist with this convention
when we compute the normalizing constant $Z_{t_{1},\ld s_{m}}$ too.)
We will now use some elementary facts from the theory of symmetric
functions to simplify the result (see \cite{StanleyVol2} or \cite{0387950672}).
Firstly, we use the following identity for the complete homogenous
symmetric functions, specialized to the exponential specialization
of parameter $t$, namely:
\[
h_{n}(\rho_{t})=\frac{t{}^{n}}{n!}.
\]
With this in hand, we notice that $W_{t}^{+}(x,y)=h_{y-x}\left(\rho_{t}\right)$
and $W_{s}^{-}(x,y)=h_{x-y}(\rho_{t})$. Hence by the Jacobi-Trudi
identity (see again \cite{StanleyVol2} or \cite{0387950672}) we
have:
\begin{eqnarray*}
\det\left(W_{t}^{+}(x_{i-1,a},x_{i,b})\right)_{1\le a,b\leq N} & = & \det\left(h_{x_{i,b}-x_{i-1,a}}\left(\rho_{t}\right)\right)_{1\leq a,b\leq N}\\
 & = & \det\left(h_{\left(\la_{b}^{(i)}-b\right)-\left(\la_{a}^{(i-1)}-a\right)}\left(\rho_{t}\right)\right)_{1\leq a,b\leq N}\\
 & = & s_{\la^{(i)}/\la^{(i-1)}}(\rho_{t}).
\end{eqnarray*}
Similarly, we have
\[
\det\left(W_{t}^{+}(x_{i-1,a},x_{i,b})\right)_{1\le a,b\leq N}=s_{\mu^{(i)}/\mu^{(i-1)}}\left(\rho_{t}\right).
\]
Thus
\begin{eqnarray*}
 &  & \p\left(\bigcap_{i=1}^{n}\left(\bigcap_{j=1}^{N}\left\{ A(j;t_{i})=x_{i,j}\right\} \right)\cap\bigcap_{i=1}^{m}\left(\bigcap_{j=1}^{N}\left\{ A(j;s_{i})=y_{i,j}\right\} \right)\right)\\
 & = & Z_{t_{1},\ld t_{n}s_{1},\ld s_{m}}^{-1}\left(\prod_{i=0}^{n}s_{\la^{(i+1)}/\la^{(i)}}\left(\rho_{t_{i+1}-t_{i}}\right)\right)\cdot\left(\prod_{i=0}^{m}s_{\mu^{(i)}/\mu^{(i+1)}}\left(\rho_{s_{i+1}-s_{i}}\right)\right).
\end{eqnarray*}

We now recognize from the statement of Theorem \ref{SchurProcessThm}, that this is exactly the probability of the Young diagram process $\la_{\tilde{L},\tilde{R}}$
passing through the Young diagrams $\left\{ \la^{(i)}\right\} _{1\leq i\leq n}$
and $\left\{ \mu^{(j)}\right\} _{1\leq j\leq m}$ at the appropriate
times except for the constant factor
of $Z_{t_{1},\ld t_{n}s_{1},\ld s_{m}}^{-1}\exp\left(-\th^{2}\right)$
. By the construction of the non-intersecting line ensemble in
terms of the Young diagram process, this is exactly the same as the
first $N$ lines of the non-intersecting line ensemble hitting the
targets $\left\{ x_{i,j}\right\} $ and $\left\{ y_{i,j}\right\} $
at the appropriate times and, since these Young diagrams have
at most $N$ non-empty rows, the remaining rows must be trivial: 

\begin{eqnarray*}
 &  & Z_{t_{1},\ld t_{n}s_{1},\ld s_{m}}^{-1} \left(\prod_{i=0}^{n}s_{\la^{(i+1)}\backslash\la^{(i)}}\left(\rho_{t_{i+1}-t_{i}}\right)\right)\cdot\left(\prod_{i=0}^{m}s_{\mu^{(i)}\backslash\mu^{(i+1)}}\left(\rho_{s_{i+1}-s_{i}}\right)\right)\\
 & = & \p\left(\bigcap_{i=1}^{n} \left\{ \la_{\tilde{L},\tilde{R}}(t_{i})=\la^{(i)}\right\} \cap\left\{ \la_{\tilde{L}.\tilde{R}}(0)=\nu\right\} \cap\bigcap_{j=1}^{m}\left\{ \la_{\tilde{L},\tilde{R}}(s_{j})=\mu^{(j)}\right\} \right)\\
 & = & \p\left(\bigcap_{i=1}^{n}\left(\bigcap_{j=1}^{N}\left\{ M_{\tilde{L},\tilde{R}}(j;t_{i})=x_{i,j}\right\} \right)\cap\bigcap_{i=1}^{m}\left(\bigcap_{j=1}^{N}\left\{ M_{\tilde{L},\tilde{R}}(j;s_{i})=y_{i,j}\right\} \right)\cap\left\{ M_{\tilde{L},\tilde{R}}(k;\cdot)=-k\ \forall k>N\right\} \right).
\end{eqnarray*}
The constant $Z_{t_{1},\ld t_{n}s_{1},\ld s_{m}}$ can be calculated
as a sum over all possible paths the non-crossing arches can take.
By our above calculation, this is the following sum over all possible
sequences of Young diagrams $\left\{ \al^{(i)}\right\} _{i=1}^{n+1}\subset\bY(N)$,
$\left\{ \be^{(i)}\right\} _{i=1}^{m+1}\subset\bY(N)$, $\nu=\al^{(n+1)}=\be^{(m)}$,
which have at most $N$ non-empty rows:
\begin{eqnarray*}
Z_{t_{1},\ld t_{n}s_{1},\ld s_{m}} & = & \sum_{\left\{ \left\{ \al^{(i)}\right\} ,\left\{ \be^{(j)}\right\} \right\} }\left(\prod_{i=0}^{n+1}s_{\al^{(i)}/\al^{(i-1)}}(\rho_{t_{i+1}-t_{i}})\right)\left(\prod_{i=0}^{m+1}s_{\be^{(i)}/\be^{(i-1)}}(\rho_{s_{i+1}-s_{i}})\right).
\end{eqnarray*}
Again, by Theorem \ref{SchurProcessThm}, except up to a constant
factor $\exp\left(-\th^{2}\right)$, this can be interpreted as a
probability for the Young diagram process $\la_{\tilde{L},\tilde{R}}$
or the line ensemble $M_{\tilde{L},\tilde{R}}$. Because we sum over
all possibilities for the first $N$ rows, we remain only with the
probability that the Young diagram process $\la_{\tilde{L},\tilde{R}}$
never has more than $N$ non-trivial rows, or equivalently that all
the line ensemble remains still for all $k>N$:
\begin{eqnarray*}
Z_{t_{1},\ld t_{n}s_{1},\ld s_{m}} & \propto & \sum_{\left\{ \left\{ \al^{(i)}\right\} ,\left\{ \be^{(j)}\right\} \right\} }\p\left(\bigcap_{i=1}^{n}\left\{ \la_{\tilde{L},\tilde{R}}(t_{i})=\al^{(i)}\right\} \bigcap\left\{ \la_{\tilde{L}.\tilde{R}}(0)=\nu\right\} \bigcap_{j=1}^{m}\left\{ \la_{\tilde{L},\tilde{R}}(s_{j})=\be^{(j)}\right\} \right)\\
 & = & \p\left(\la_{\tilde{L},\tilde{R}}(\cdot)\text{ has at most }N\text{ non-empty rows}\right)\\
 & = & \p\left(M_{\tilde{L},\tilde{R}}(k;\cdot)=-k\ \forall k>N\right).
\end{eqnarray*}
Combining the two calculations, we see that the two factors of $\exp\left(-\th^{2}\right)$
cancel and we remain with:
\begin{eqnarray*}
 &  & \p\left(\bigcap_{i=1}^{n}\left(\bigcap_{j=1}^{N}\left\{ A(j;t_{i})=x_{i,j}\right\} \right)\cap\bigcap_{i=1}^{m}\left(\bigcap_{j=1}^{N}\left\{ A(j;s_{i})=y_{i,j}\right\} \right)\right)\\
 & = & \frac{\p\left(\bigcap_{i=1}^{n}\left(\bigcap_{j=1}^{N}\left\{ M_{\tilde{L},\tilde{R}}(j;t_{i})=x_{i,j}\right\} \right)\cap\bigcap_{i=1}^{m}\left(\bigcap_{j=1}^{N}\left\{ M_{\tilde{L},\tilde{R}}(j;s_{i})=y_{i,j}\right\} \right)\cap\left\{ M_{\tilde{L},\tilde{R}}(k;\cdot)=-k\ \forall k>N\right\} \right)}{\p\left(M_{\tilde{L},\tilde{R}}(k;\cdot)=-k\ \forall k>N\right)}\\
 & = & \p\left(\bigcap_{i=1}^{n}\left(\bigcap_{j=1}^{N}\left\{ M_{\tilde{L},\tilde{R}}(j;t_{i})=x_{i,j}\right\} \right)\cap\bigcap_{i=1}^{m}\left(\bigcap_{j=1}^{N}\left\{ M_{\tilde{L},\tilde{R}}(j;s_{i})=y_{i,j}\right\} \right)\given{M_{\tilde{L},\tilde{R}}(k;\cdot)=-k\ \forall k>N}\right),
\end{eqnarray*}
as desired.
\end{proof}

\section{Relationship to Stochastic Dynamics on Partitions}

In this section we show that the Poissonized RS process can
be understood as a special case of certain stochastic dynamics on
partitions introduced by Borodin and Olshanski in \cite{Borodin_stochasticdynamics}. 
\begin{thm}
\label{curve_thm}Let $\left(u(t),v(t)\right)$, $t\in[-\th,\th]$
be the parametric curve in $\bR_{+}^{2}$ given by:
\[
u(t)=\begin{cases}
\th & t\leq0\\
\th-\abs t & t\geq0
\end{cases}\ ,\ v(t)=\begin{cases}
\th-\abs t & t\leq0\\
\th & t\geq0
\end{cases}.
\]
Also let $\square(u(t),v(t))$ be the rectangular region $[0,u(t)]\times[0,v(t)]\subset\bR_{+}^{2}$. (See Figure \ref{fig:curve_cartoon} for an illustration of this setup).

For any point configuration $\Pi\in\cC^{\th}$, let $\pi(t)$ be the
permutation associated with the point configuration $\Pi\cap\mbox{\ensuremath{\square}}[u(t),v(t)]$
(as in Definition \ref{associatedpermutation}) and let $\la_{\Pi}(t)=\sh\left(RS(\pi(t)\right)$
be the \uline{shape} of the Young tableau that one gets by applying
the ordinary RS bijection to the permutation $\pi(t)$. 

If $(\tilde{L},\tilde{R})=dRS(\Pi)$ is the decorated Young tableau
that one gets by applying the decorated RS bijection to the configuration
$\Pi$, then the Young diagram process of $(\tilde{L},\tilde{R})$
is exactly $\la_{\Pi}(t)$: 
\[
\la_{\tilde{L},\tilde{R}}(t)=\la_{\Pi}(t)\ \forall t\in[-\th,\th]
\].
\end{thm}

\begin{proof}
For concreteness, let us suppose there are $n$ points in the configuration
$\Pi$ and label them $\left\{ (x_{i},y_{i})\right\} _{i=1}^{n}$
sorted in ascending order of $x$-coordinate. Also label $\si=\pi(0)\in S_{n}$
be the permutation associated to the configuration $\Pi$ and let
$(\tilde{L},\tilde{R})=\left(L,(y_{\si^{-1}(1)},\ld,y_{\si^{-1}(n)}),R,\left(x_{1},\ld,x_{n}\right)\right)$
be the output of the decorated RS correspondence. 

We prove first the case $t\in[0,\th]$, then use a symmetry property
of the RS algorithm to deduce the result for $t\in[-\th,0]$. Fix
a $k$ with $1\leq k\leq n$ and let us restrict our attention to
times $t\in[0,\th]$ for which $x_{k-1}<\th-t\leq x_{k}$. By the
definition of the Young diagram process, we have then by this choice
of $t$ that
\begin{eqnarray*}
\la_{\tilde{L},\tilde{R}}(t) & = & \left\{ (i,j):x_{R(i,j)}\leq\th-t\right\} \\
 & = & \left\{ (i,j):R(i,j)\leq k\right\}. 
\end{eqnarray*}
Now, by the definition of the decorated RS correspondence, the pair
of tableaux $(L,R)$ correspond to the permutation $\si$ when one
applies the ordinary RS algorithm. Since $R$ is the recording tableau
here, the set $\left\{ (i,j):R(i,j)\leq k\right\} $ is exactly the
shape of the tableaux in the RS algorithm after $k$ steps of the
algorithm. At this point the algorithm has used only comparisons between
the numbers $\si(1),\si(2),\ld,\si(k)$; it has not seen any other
numbers yet.

On the other hand, we have $\mbox{\ensuremath{\square}}\left(u(t),v(t)\right)=\left\{ (x_{i},y_{i})\in\Pi:\ x_{i}<\th-t\right\} =\left\{ (x_{i},y_{i})\right\} _{i=1}^{k}$
by the choice $x_{k-1}<\th-t\leq x_{k}$ and since $x_{i}$ are sorted.
So $\la_{u,v}(t)=\sh\left(RSK(\pi(t)\right)$ is the shape outputted
by the RS algorithm after it has worked on the permutation $\pi(t)\in S_{k}$
using comparisons between the numbers $\pi(t)(1),\pi(t)(2),\ld,\pi(t)(k)$. 

But we now notice that for $1\leq i,j\leq k$ that $\pi(t)(i)<\pi(t)(j)$
if and only $\si(i)<\si(j)$ since they both happen if and only if
$y_{i}<y_{j}$. Hence, in computing $\la_{u,v}(t)$, the RS algorithm
makes the exact same comparisons as the first $k$ steps of the RS
algorithm on the list $\si(1),\si(2),\ld,\si(k)$. For this reason,
$\la_{\Pi}(t)=\left\{ (i,j):R(i,j)\leq k\right\} $. Hence $\la_{\tilde{L},\tilde{R}}(t)=\la_{\Pi}(t)$,
as desired. Since this works for any choice of $k$, this covers all
of $t\in[0,\th]$

To handle $t\in[-\th,0]$ consider as follows. Let $\Pi^{T}$ be the
reflection of the point configuration $\Pi$ about the line $x=y$,
in other words swapping the $x$ and $y$ coordinates of every point.
Then the permutation associated with $\Pi^{T}$ is $\si^{-1}$. Using
the remarkable fact of the RS correspondence $RS(\si)=(L,R)\iff RS(\si^{-1})=(R,L)$
we will have from the definition that $dRS(\Pi^{T})=(\tilde{R},\tilde{L})=\left(R,\left(x_{1},\ld,x_{n}\right),L,(y_{\si^{-1}(1)},\ld,y_{\si^{-1}(n)})\right)$. 

Using the result for $t\in[0,\th]$ applied to the configuration $\Pi$,
we have $\la_{\tilde{R},\tilde{L}}(t)=\la_{\Pi^{T}}(t)$ for all $t\in[0,\th]$.
Now, $\la_{\tilde{R},\tilde{L}}(t)=\la_{\tilde{L},\tilde{R}}(-t)$
follows from the definition \ref{YDprocess_pair}. It is also true
that $\la_{\Pi^{T}}(t)=\la_{\Pi}(-t)$; this follows from the fact
that $\left(\mbox{\ensuremath{\square}}\left(u(t),v(t)\right)\right)^{T}=\mbox{\ensuremath{\square}}\left(u(-t),v(-t)\right)$
as regions in the plane $\bR_{+}^{2}$ and so the permutations at
time $t$ will have $\si_{\Pi}(t)=\left(\si_{\Pi^{-1}}(-t)\right)^{-1}$.
Since the RS correspondence assigns the same shape to the inverse
permutation have $\la_{\Pi^{T}}(t)=\la_{\Pi}(-t)$. Hence we conclude
$\la_{\tilde{L},\tilde{R}}(t)=\la_{\Pi}(t)$ for $t\in[-\th,0]$ too.
\end{proof}

\begin{figure}
\centering
\begin{subfigure}[b]{.3\textwidth}
\begin{center}
\begin{tikzpicture}[scale = 2.5]
			\coordinate (BL) at (0,0);
			\coordinate (BR) at (1,0);
			\coordinate (TL) at (0,1);
			\coordinate (TR) at (1,1);
			\coordinate (Xmax) at (1.3,0);
			\coordinate (Ymax) at (0,1.3);
			\coordinate (P) at (1.0,0.6);

			\fill[fill=gray!40!white] (0,0) rectangle (P)
				node[anchor=west]{$\big(u(t),v(t)\big)$};

\fill (P) circle[radius=0.5pt];
							
			\node[anchor=west] at (TR) {$(\theta,\theta)$};
				
			\draw [black, ->, thick] (BL) -- (Xmax);
			\draw [black, ->, thick] (BL) -- (Ymax);
			\draw [black] (TR) -- (BR)
				 node[anchor=north]{$(\theta,0)$};			
			\draw [black] (TR) -- (TL)
				 node[anchor=east]{$(0,\theta)$};
			
		\end{tikzpicture}
		\end{center}
		\caption{ $t <0 $}
\end{subfigure}		
\hspace{0.15\textwidth}
\begin{subfigure}[b]{.3\textwidth}
\begin{center}
\begin{tikzpicture}[scale = 2.5]
			\coordinate (BL) at (0,0);
			\coordinate (BR) at (1,0);
			\coordinate (TL) at (0,1);
			\coordinate (TR) at (1,1);
			\coordinate (Xmax) at (1.3,0);
			\coordinate (Ymax) at (0,1.3);
			\coordinate (P) at (0.7,1.0);

			\fill[fill=gray!40!white] (0,0) rectangle (P)
				node[anchor=south]{$\big(u(t),v(t)\big)$};

\fill (P) circle[radius=0.5pt];
							
			\node[anchor=west] at (TR) {$(\theta,\theta)$};
				
			\draw [black, ->, thick] (BL) -- (Xmax);
			\draw [black, ->, thick] (BL) -- (Ymax);
			\draw [black] (TR) -- (BR)
				 node[anchor=north]{$(\theta,0)$};			
			\draw [black] (TR) -- (TL)
				 node[anchor=east]{$(0,\theta)$};
			
		\end{tikzpicture}
		\end{center}
		\caption{ $t  > 0 $}
\end{subfigure}
		
		\caption{ The point $\big( u(t),v(t) \big), t\in [-\th,\th]$ moves ``counterclockwise'' around the outer boundary of the square $[0,\th]\times[0,\th]$, starting from $(\th,0)$ at $t = -\th$, then moving to $(\th,\th)$ at $t=0$, and finally moving to $(0,\th)$ at $t = \th$. The region $\square\big(u(t),v(t)\big)$ is the shaded rectangular area bounded between $(0,0)$ and $\big( u(t),v(t) \big)$.}
		
		\label{fig:curve_cartoon}
\end{figure}
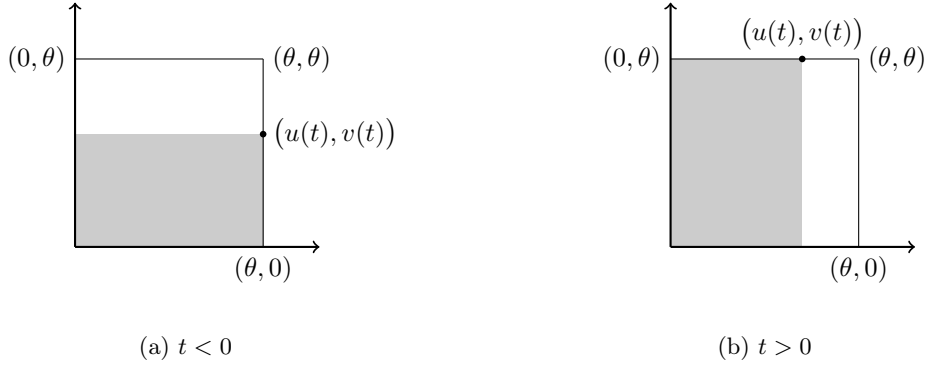

\begin{rem}
This is exactly the same construction of the random trajectories $\la_{\Pi}(u(t),v(t))$
from Theorem 2.3 in \cite{Borodin_stochasticdynamics}. Note that
in this work, the curve $\left(u(t),v(t)\right)$ was going the other
way going ``clockwise'' around the outside edge of the box $[0,\th]\times[0,\th]$
rather than ``counterclockwise'' as we have here. This difference
just arises from the convention of putting the recording tableau as
the right tableau when applying the RS algorithm and makes no practical
difference. Since our construction is a special case of the stochastic
dynamics constructed from this paper, we can use the scaling limit
results to compute the limiting behavior of the Poissonized RS
tableaux. The only obstruction is that one has to do some change of
time coordinate to translate to what is called ``interior time''
in \cite{Borodin_stochasticdynamics} along the curve so that $s=\half\left(\ln u-\ln v\right)$.
In the below corollary, we record the scaling limit for the topmost
line of the ensemble. \end{rem}
\begin{cor}
Let $\left(\tilde{L}_{\th},\tilde{R}_{\th}\right)$ be the Poissonized
Robinson-Schensted tableaux of parameter $\th$. If we scale around
the point $\tau=0$, there is convergence to the Airy 2 process on
the time scale $\th^{2/3}$, namely:
\[
\frac{M_{\tilde{L}_{\th},\tilde{R}_{\th}}\left(1;2\th^{2/3}\tau\right)-\left(2\th-2\th^{2/3}\abs{\tau}\right)}{\th^{1/3}}+\tau^{2}\To\cA_{2}(\tau),
\]
where $\cA_{2}(\cdot)$ is the Airy 2 process.\end{cor}
\begin{proof}
We first do a change of variables in the parameter $\th$ by $\al=\th^{2}$
so that the curves we consider have area $u_{\al}(0)v_{\al}(0)=\al$
at time $0$. We will use $s$ to denote ``interior time'' along
the curve $(u_{\al},v_{\al})$ constructed in Theorem \ref{curve_thm}
(see Remark 1.4 in \cite{Borodin_stochasticdynamics} for an explanation
of the interior time). This is a change of time variable $s=s(t)$
given by
\[
s(t)=\half\left(\ln u_{\al}(t)-\ln v_{\al}(t)\right)=\half\sgn\left(t\right)\ln\left(1-\frac{\abs t}{\sqrt{\al}}\right).
\]
Notice by Taylor expansion now that $s(2\al^{1/3}\tau)=-\tau\al^{-1/6}+o(\al^{-1/6})$
as $\al\to\infty$. By application of Theorem 4.4. from \cite{Borodin_stochasticdynamics},
(see also Corollary 4.6. for the simplification of Airy line ensemble
to the Airy process) we have that as $\al\to\infty$ that
\[
\frac{M_{\tilde{L}_{\al},\tilde{R}_{\al}}\left(1;2\al^{1/3}\tau\right)-2\sqrt{u_{\al}(2\al^{1/3}\ta)v_{\al}(2\al^{1/3}\tau)}}{\al^{1/6}}\To\cA_{2}(\tau).
\]
Doing a taylor expansion now for $2\sqrt{u_{\al}(t)v_{\al}(t)}=2\sqrt{\al}\sqrt{1-\frac{\abs t}{\sqrt{\al}}}$
gives
\begin{eqnarray*}
2\sqrt{u_{\al}(2\al^{1/3}\ta)v_{\al}(2\al^{1/3}\tau)} & = & 2\sqrt{\al}-2\al^{1/3}\abs{\tau}-\al^{1/6}\tau^{2}+o(\al^{1/6}).
\end{eqnarray*}
Putting this into the above convergence to $\cA_{2}$ gives the desired
result.\end{proof}
\begin{rem}
The scaling that is needed for the convergence of the top line to
the Airy 2 process here is exactly the same as the scaling that appears
for a family of non-crossing Brownian bridges to converge to the Airy
2 process, see \cite{Corwin_browniangibbs}. This is not entirely
surprising in light of Theorem 3.2, which shows that $M_{\tilde{L}_{\th},\tilde{R}_{\th}}$
is related to a family of non-crossing Poisson arches, and it might
be expected that non-crossing Poisson arches have the same scaling
limit as non-crossing Brownian bridges. In this vein, we conjecture that it is also possible to get a convergence result for the whole ensemble (not just the topmost line) to the multi-line Airy 2 line ensemble, in the sense of weak convergence as a line ensemble introduced in Definition 2.1 of \cite{Corwin_browniangibbs}.
\end{rem}

\section{A discrete limit}
The Poissonized RS process can be realized as the limit of a discrete model created from geometric random variables in a certain scaling limit. This discrete model is a special case of the corner growth model studied in Section 5 of \cite{JoDPP}. We will present the precise construction of the model here, rather than simply citing \cite{JoDPP} in order to present it in a way that makes the connection to the Poissonized RS tableaux more transparent. We also present a different argument yielding the distribution of the model here, again to highlight the connection to the Poissonized RS tableaux. Our proof is very different than the proof from \cite{JoDPP}; it has a much more probabilistic flavor closer to the proof of Theorem \ref{SchurProcessThm}. 
 
  One difference between the discrete model and the Poissonized RS process is due to the possibility of multiple points with the same x-coordinate or y-coordinate. (These events happen with probability 0 for the Poisson point process.) To deal with this we must use Robinson-Schensted-Knuth (RSK) correspondence, which generalizes the RS correspondence to a bijection from generalized permutations to semistandard Young tableau (SSYT). See Section 7.11 in \cite{StanleyVol2} for a reference on the RSK correspondence.

\subsection{Discrete Robinson-Schensted-Knuth process with geometric weights}

\begin{defn}
Fix a parameter $\theta\in\bR_{+}$ and an integer $k \in \bN$. Let $\cL^{\th,k} = \{\th/k,2\th/k \ld \th \}$ be a discretization of the interval $[0,\th]$ with $k$ points. Let $T$ be any semistandard Young tableau (SSYT) whose entries do not exceed $k$. The \textbf{$\cL^{\th,k}$-discretized Young diagram process} for $T$ is a Young diagram valued map $\la^{\th,k}_T:\bR_+ \to \bY$ defined by
\[
\la^{\th,k}_T(t) = \left\{ (i,j) : T(i,j)\frac{\th}{k} \leq t \right\}.
\]
If we are given two such SSYT $L,R$ so that $sh(L) = sh(R)$, we can define the $\cL^{\th,k}$-discretized Young diagram process $\la^{\th,k}_{L,R}: [-\th,\th] \to \bY$ for the pair $(L,R)$ by

\[
\la^{\th,k}_{L,R}(t)= \begin{cases}
\la^{\th,k}_{L}(\th+t) & \ t\leq 0 \\
\la^{\th,k}_{R}(\th-t) & \ t\geq 0
\end{cases}.
\]
\end{defn}

\begin{rem}
\label{discrem}
This definition is analogous to Definition \ref{YDprocess} and Definition \ref{YDprocess_pair}. Comparing with Definition \ref{YDprocess}, we see that in the language of decorated Young tableau, the $\cL^{\th,k}$-discretized Young diagram process corresponds to thinking of decorating the Young tableau with a decoration of $t_{T(i,j)} = T(i,j) \frac{\th}{k}$. For example the SSYT $\young(122,34)$ would be represented graphically as: (compare with Example \ref{YTexample})

\begin{center}
\ytableausetup{centertableaux,boxsize=2.5em} 
\begin{ytableau} 
\stackrel{1}{ \thk } & \stackrel{2}{ 2 \thk} & \stackrel{2}{2 \thk} \\
\stackrel{3}{3 \thk} & \stackrel{4}{4 \thk}
\end{ytableau}.
\end{center}
In other words, the decorations are proportional to the entries in the Young tableau by a constant of proportionality $\slfrac{\th}{k}$. This scaling of the $\cL^{\th,k}$-discretized Young diagram process, despite being a very simple proportionality, will however be important to have convergence to the earlier studied Poissonized RS model in a limit as $k \to \infty$.
\end{rem}

\begin{defn}
Let $\cC_{n}^{\th,k}$ be the set of $n$ point configurations on the lattice $\cL^{\th,k} \times \cL^{\th,k} \subset\bR_{+}^{2}$ where we allow the possibility of multiple points to sit at each site. Since there are only $k^2$ possible locations for the points, one can think of elements of  $\cC_{n}^{\th,k}$ in a natural way as $\bN$-valued $k \times k$ matrices, whose entries sum to $n$:
\[
\cC_{n}^{\th,k} = \left\{ \left\{\pi_{a,b}\right\}_{a,b =1 }^k : \sum_{a,b = 1}^{k} \pi_{a,b} = n \right\}.
\]
Let $\cT_{n}^{\th,k}$ be the set of pairs of semistandard tableaux of size $n$ and of the same shape whose entries do not exceed $k$. This is
\[
\cT_{n}^{\th,k} = \left\{ (L,R) : \sh(L) = \sh(R), \abs{L} = \abs{R} = n, L(a,b) \leq k, R(a,b) \leq k \, \forall a,b \right\}.
\]
The Robinson-Schensted-Knuth (RSK) correspondence is a bijection between $\bN$-valued matrices (or equivalently generalized permutations) and pairs of semistandard tableaux of the same shape. Thinking of $\cC^{\th,k}_n$ as $\bN$-valued matrices, we see more precisely that the RSK correspondence is a bijection between $\cC_{n}^{\th,k}$ and $\cT_{n}^{\th,k}$. Composing this with the definition of the $\cL^{\th,k}$-discretized Young diagram process we have a bijection, which we call the \textbf{ $\cL^{\th,k}$-discretized RSK bijection} between configurations in $\cC_{n}^{\th,k}$ and $\cL^{\th,k}$-discretized Young diagram processes $\la^{\th,k}_{L,R}: [-\th,\th] \to \bY$. We will also use the shorthand $\cC^{\th,k} = \bigcup_n \cC^{\th,k}_n$ and $\cT^{\th,k} = \bigcup_n \cT^{\th,k}_n$.
\end{defn}

\begin{defn}
 Let $\{ \xi_{i,j}: 1\leq i\leq k, 1\leq j \leq k \}$ be an iid collection of  geometric random variables with parameter $\th^2/k^2$. To be precise, each $\xi$ has the following probability mass function:
 
 \[
 \p(\xi = x) = \left(1 - \frac{\th^2}{k^2}\right)\left( \frac{\th^2}{k^2} \right)^x.
 \]
This gives a probability measure on the set of point configurations $\cC^{\th,k}$ by placing exactly $\xi_{i,j}$ points at the location $\left(i \slfrac{\th}{k}, j \slfrac{\th}{k}\right)$. By applying the $\cL^{\th,k}$-discretized RSK bijection, this induces a probability measure on $\cL^{\th,k}$-discretized Young diagram processes.  We refer to the resulting pair of random semistandard tableaux $(L,R)$ as the \textbf{ $\cL^{\th,k}$-geometric weight RSK tableaux}, and we refer to the Young diagram process $\la^{\th,k}_{L,R}$, as the \textbf{ $\cL^{\th,k}$-geometric weight RSK process}.
\end{defn}

\begin{rem}
The word ``geometric weight'' is always in reference to the distribution of the variables $\xi$. This should \emph{not} be confused with the ``Geometric RSK correspondence'', as in \cite{OtherGeoRSK}, which is a different object and in which ``geometric'' refers to a geometric lifting. 
\end{rem}

\begin{rem}
It is possible to construct similar models where the parameter of the geometric random variable used to place particles differs from site to site. For our purposes, however, we will stick to this simple case where all are equal to make the construction and the convergence to the Poissonized RS process as clear as possible. See Section 5 of \cite{JoDPP} for a more general treatment. 
\end{rem}

With this set up, we have the following very close analogue of Theorem \ref{SchurProcessThm} for the $\cL^{\th,k}$-geometric weight RSK process, which characterizes the law of this random object.

\begin{thm}
\label{DiscreteeSchurProcessThm}Fix $\th>0$, $k \in \bN$ and times $t_{1}<\ld<t_{n}\in[-\th,0]$
and $s_{1}<\ld<s_{m}\in[0,\th]$. Suppose we are given an increasing
list of Young diagrams $\la^{(1)}\subset\ld\subset\la^{(n)}$, a decreasing
list of Young diagrams $\mu^{(1)}\supset\ld\supset\mu^{(m)}$ and
a Young diagram $\nu$ with $\nu\supset\la^{(n)}$ and $\nu\supset\mu^{(1)}$.
To simplify the presentation we will use the convention $\la^{(0)}=\emptyset,\la^{(n+1)}=\nu,\mu^{(0)}=\nu,\mu^{(m+1)}=\emptyset$
and $t_{0}=-\th,t_{n+1}=0,s_{0}=0,s_{m+1}=\th$. The geometric weight RSK process $\la^{\th,k}_{L,R}$ has the following finite dimensional distribution:
\begin{eqnarray*}
 &  & \p\left( \bigcap_{i=1}^{n}\left\{ \la^{\th,k}_{L,R}(t_{i})=\la^{(i)}\right\} \bigcap\left\{ \la^{\th,k}_{L,R}(0)=\nu\right\} \bigcap\bigcap_{j=1}^{m}\left\{ \la^{\th,k}_{L,R}(s_{j})=\mu^{(j)}\right\} \right)\\
 & = & \left(1 - \frac{\th^2}{k^2}\right)^{k^2} \left(\frac{\th^2}{k^2}\right)^{\abs{\nu}} \left(\prod_{i=0}^{n} \Dim _{\cL^{\th,k}(t_{i+1},t_{i})} \left(\la^{(i+1)}/\la^{(i)}\right)\right)  \cdot\left(\prod_{i=0}^{m} \Dim _{\cL^{\th,k}(s_{i+1},s_{i})} \left(\mu^{(i)}/\mu^{(i+1)}\right)\right),
\end{eqnarray*}
where $\Dim_k(\la/\mu)$ is the number of SSYT of skew shape $\la/\mu$ and whose entries do not exceed $k$ and $\cL^{\th,k}(x,y) = \abs{\cL^{\th,k}\cap(x,y]}$ is the number of discretization points from $\cL^{\th,k}$ in the interval $(x,y]$.
\end{thm}

\begin{rem}
\label{DicreteSchurProcessRem}
The above theorem is purely combinatorial in terms
of $\Dim_j\left(\la/\mu\right)$, which is the enumerating semistandard Young
tableaux. As was the case for Theorem \ref{SchurProcessThm}, this can be written in a very nice way using Schur functions and specializations. Let $\sigma^{\th,k}_{x,y}$ to be the specialization that specializes the first $\cL^{\th,k}_{(x,y)}$ variables to $\th/m$ and the rest to zero. Namely:
\[
f(\sigma^{\th,k}_{x,y}) = f\left(\underbrace{\frac{\th}{k},\frac{\th}{k},\ld,\frac{\th}{k}}_{ \cL^{\th,k}_{(x,y)}}, 0, 0, \ld \right).
\]
This specialization differs by a constant factor from the so called ``principle specialization'', see Section 7.8 of \cite{StanleyVol2}. It is an example of a ``finite length'' specialization as defined in Section 2.2.1 in \cite{MacProc}. One has the identity:
\[
s_{\la/\mu}\left(\sigma^{\th,k}_{x,y}\right)= \left(\frac{\th}{k}\right)^{\abs{\la/\mu}} \Dim_{\cL^{\th,k}_{(x,y)}} (\la/\mu).
\]
Plugging this into the above theorem, after some very nice telescoping cancellations, we can rewrite the probability as a chain of Schur functions:
\begin{eqnarray*}
 &  & \p\left(\bigcap_{i=1}^{n}\left\{ \la^{\th,k}_{L,R}(t_{i})=\la^{(i)}\right\} \bigcap\left\{ \la^{\th,k}_{L,R}(0)=\nu\right\} \bigcap\bigcap_{j=1}^{m}\left\{ \la^{\th,k}_{L,R}(s_{j})=\mu^{(j)}\right\} \right)\\
 & = & \left(1 - \frac{\th^2}{k^2}\right)^{k^2} \left(\prod_{i=0}^{n}s_{\la^{(i+1)}/\la^{(i)}}\left(\sigma^{\theta,N}_{t_{i+1},t_{i}}\right)\right)\cdot\left(\prod_{i=0}^{m}s_{\mu^{(i)}/\mu^{(i+1)}}\left(\sigma^{\theta,N}_{s_{i+1},s_{i}}\right)\right).
\end{eqnarray*}
\end{rem}
\begin{rem}
For fixed $\th$, one might notice that the normalizing prefactor of the geometric weight RSK process $(1-\th^2/k^2)^{k^2}$ converges as $k\to \infty$ to  $e^{-\th^2}$, the normalizing prefactor of the Poissonized RS process. Even more remarkably, the specializations $\sigma^{\th,k}_{x,y}$ that appear in Remark \ref{DicreteSchurProcessRem}, converge to the specialization $\rho_{y - x}$ that appear in Remark \ref{ShurProcRem} in the sense that for any symmetric function $f$, one has that
\[
\lim_{N \to \infty} f(\sigma^{\th,N}_{y,x}) = f(\rho_{y-x}). 
\]
One can verify this convergence by checking the effect of the specialization on the basis $p_\la$ of power sum symmetric functions. These have
\[
p_\la(\rho_{y-x}) =\begin{cases}
(y-x)^n & \ \la = (\underbrace{1,1,\ld,1}_n) \\
0 & \ \text{otherwise}
\end{cases},
\]
and for $\sigma^{\th,N}_{x,y}$, we have 
\[
p_{\la}(\sigma^{\th,N}_{x,y}) = \left(\cL^{\th,k}_{(x,y)}\right)^{\ell(\la)}\cdot \left( \frac{\th}{k} \right)^{\abs{\la}}.
\]
(Here $\ell(\la)$ is the number of rows of $\la$.) Using the bound, $\floor{(y-x)k/\th} \leq \cL^{\th,k}_{(x,y)} \leq \ceil{(y-x)k/\th}$, we know that $\cL^{\th,k}_{(x,y)}$ differs from $(y-x)k/\th$ by no more than one. Since $\ell(\la) \leq \abs{\la}$ holds for any Young diagram, the above converges to $0$ as $k \to \infty$ unless $\ell(\la) = \abs{\la}$. This only happens if $\la = (\underbrace{1,1,\ld,1}_n)$ is a single vertical column and in this case we get exactly a limit of $(y-x)^n$ as $k \to \infty$, which agrees with $p_\la(\rho_{y-x})$. See section 7.8 of \cite{StanleyVol2} for these formulas. This convergence of finite length specializations to Plancherel specializations is also mentioned in Section 2.2.1. of \cite{MacProc}.

This observations shows us that the finite dimensional distributions of the geometric weight RSK process converge to the finite dimensional distributions of the Poissonized RS process in the limit $k \to \infty$. One might have expected this convergence since the point process of geometric points from which the geometric weight RSK process is built convergences in the limit $k \to \infty$ to a Poisson point process of rate 1 in the square $[0,\th]\times[0,\th]$, from which the Poissonized RS process is built. However, since the decorated RS correspondence can be very sensitive to moving points even very slightly, it is not apriori clear that convergence of point processes in general always leads to convergence at the level of Young diagram processes. 
\end{rem}

\subsection{Proof of Theorem \ref{DiscreteeSchurProcessThm}}

The proof follows by similar methods to the proof of Theorem \ref{SchurProcessThm}. We prove some intermediate results which are the analogues of Lemma \ref{conditioninglemma}, Corollary \ref{YTcor} and Lemma \ref{sizelemma}.

\begin{lem}
Let $N=\abs{\la^{\th,k}_{L,R}(0)}$.
Then $N$ has the distribution of the sum of $k^2$ i.i.d. geometric random variable of parameter $\th^2 / k^2$. Moreover, conditioned on the event $\left\{ N=n\right\} $ the pair $(L,R)$ is uniformly distributed in the set $\cT^{k,\th}_n$.
\end{lem}

\begin{proof}

This is analogous to Lemma \ref{sizelemma} In this case, $N = \sum_{i,j=1}^{k} \xi_{i,j}$ is the sum of geometric random variables.

The fact that all elements of $\cT_{n}^{k,\th}$ are equally likely in the conditioning $\{ N = n \}$ is because of the following remarkable fact about geometric distributions. For a collection of iid geometric random variables, the probability of any configuration $\bigcap_{i,j = 1}^{k}\left\{ \xi_{i,j} = x_{i,j} \right\}$ depends only on the \emph{sum} $\sum_{i,j = 1}^k x_{i,j}$. Indeed, when $p$ is the parameter for the geometric random variables, the probability is:

\[
\p\left(\bigcap_{i,j = 1}^{k}\left\{ \xi_{i,j} = x_{i,j} \right\}\right) = p^{\sum_{i,j = 1}^k x_{i,j}}(1-p)^{k^2}.
\]
Since this depends only on the sum, and not any other detail of the $x_{i,j}$, when one conditions on the sum, all the configurations are equally likely. Since the RSK is a bijection, it pushes forward the uniform distribution on $\cC^{\th,k}_n$ to a uniform distribution on $\cT^{\th,k}_n$ as desired.
\end{proof}

\begin{rem}
This remarkable fact about geometric random variables is the analogue of the fact that the points of a Poisson point process are uniformly distributed when one conditions on the total number of points. This was a cornerstone of Lemma \ref{conditioninglemma}. This special property of geometric random variables is what makes this distribution so amenable to analysis: see Lemma 2.2. in the seminal paper by Johansson \cite{Johansson338249} where this exact property is used.
\end{rem} 

\begin{cor}
For any Young diagram $\nu$, we have: 
\begin{eqnarray*}
\p\left(\la^{\th,k}_{L,R}(0)=\nu\right)& = &\left(\frac{\Dim_{k}(\nu)^{2}}{\binom{k^2 + \abs{\nu} - 1}{k^2} }\right)\p\left(N=\abs{\nu}\right) \\
& =& \left(1 - \frac{\th^2}{k^2}\right)^{k^2} \left(\frac{\th^2}{k^2}\right)^{\abs{\nu}} \Dim_{k}(\nu)^{2}.
\end{eqnarray*}
\end{cor}

\begin{proof}
This is analogous to the proof of Corollary \ref{YTcor}. The only difference is that $\cT^{\th,k}_n$ contains pairs of semi-standard with entries no larger than $k$, of which we are interested in the number of pairs of shape $\nu$. This is exactly what $\Dim_{k}(\nu)$ enumerates. $\abs{ \cC^{\th,k}_n} = \abs{\cT^{\th,k}_n} = \binom{k^2 + \abs{\nu} - 1}{k^2}$ is the number of elements in $\cT^{\th,k}_n$, so it appears as a normalizing factor. In this case, since $N$ has the distribution of the sum of $k^2$ geometric random variables, we can simplify using the probability mass function:
\[
\p\left( N = x \right) = \binom{k^2 + x- 1}{k^2}\left(\frac{\th^2}{k^2}\right)^x\left(1-\frac{\th^2}{k^2}\right)^{k^2}.
\]
\end{proof}

\begin{lem}
We have
\[
\p\left( \bigcap_{i=0}^{n+1} \left\{ \la^{\th,k}_{L,R}(t_{i})=\la^{(i)} \right\} \given{  \la^{\th,k}_{L,R}(0)=\nu }  \right) = \frac{\prod_{i=1}^{n+1}\Dim_{\cL^{\th,k}_{(t_{i+1},t_{i})} \left(\la^{(i+1)}/\la^{(i)}\right)}}{\Dim_{k}(\nu)}.
\]
An analogous formula holds for $\p\left(\bigcap_{i=0}^{m+1}\left\{ \la^{\th,k}_{L,R}(s_{i})=\mu^{(i)}\right\} \given{  \la^{\th,k}_{L,R}(0)=\nu }\right)$. Moreover, we have the same type of conditional independence as from Lemma \ref{shapelem} between times $t > 0$ and times $t < 0$ when we condition on the event $\left\{ \la^{\th,k}_{L,R}(0)=\nu \right\}$. 
\end{lem}
\begin{proof}
This is the analogue of Lemma \ref{shapelem}. The proof proceeds in the same way with the important observation that, when conditioned on $\left\{ \la^{\th,k}_{L,R}(0)=\nu \right\}$, the pair of SSYT $(L,R)$ is uniformly chosen from the set of pairs of shape $\nu$ from $\cT^{\th,k}_n$. This set is the Cartesian product of the set of all such SSYT of shape $\nu$ with itself. Hence, the two SSYT are independent and are both uniformly distributed among the set of all SSYT of shape $\nu$ in this conditioning. For this reason, it suffices to count the number of SSYT of shape $\nu$ with the correct intermediate shapes at times $t_1,\ld,t_n$. The counting of these SSYT then follows by the same type of argument as Lemma \ref{shapelem}. Each intermediate SSYT of shape $\mu^{i+1}/\mu^{i}$ must be filled with entries from the interval $\cL^{\th,k}\cap (t_i,t_{i+1}]$ in order for the resulting SSYT to have the correct subshapes. Since we are only interested in the number of such SSYT, counting those with entries between 1 and $\cL^{\th,k}_{t_{i+1},t_{i}}$ will do. This is precisely what $\Dim_{\cL^{\th,k}(t_{i+1},t_{i})} \left(\la^{(i+1)}/\la^{(i)}\right)$ enumerates.
\end{proof}

\begin{rem}
In the proof of Theorem \ref{SchurProcessThm}, there were additional lemmas needed to separate the dependence of the decorations and the entries appearing in the Young diagrams. As explained in Remark \ref{discrem},the discrete geometric weight RSK tableaux case is simpler in this respect because the decorations are proportional to the entries in the tableaux by a factor of $\th/k$.
\end{rem}

\begin{proof} (Of Theorem \ref{DiscreteeSchurProcessThm})
Exactly as in the proof of Theorem \ref{SchurProcessThm}, the proof follows by combining the lemmas.
\end{proof}

\textbf{Acknowledgments}
The author extends many thanks to G{\' e}rard Ben Arous for early encouragement on this subject and to Ivan Corwin for his friendly support and helpful discussions, in particular pointing out the connections that led to the development of Section 5. The author was partially supported by NSF grant DMS-1209165.

\newpage

\bibliographystyle{plain}
\bibliography{biblio3}

\end{document}